\documentclass[12pt]{article}
\usepackage{amsmath,amssymb}
\usepackage{color,graphicx}
\usepackage{stmaryrd}
\newcommand{\bs}{\boldsymbol}
\newcommand{\mbf}{\mathbf}

\newcommand{\bb}{\mathbf b}

\newcommand{\be}{\mathbf e}
\newcommand{\bE}{\mathbf E}
\newcommand{\bF}{\mathbf F}
\newcommand{\bv}{\mathbf v}

\newcommand{\bw}{\mathbf w}
\newcommand{\bu}{\mathbf u}
\newcommand{\bx}{\mathbf x}
\newcommand{\bg}{{\boldsymbol g}}
\newcommand{\bG}{{\boldsymbol G}}

\newcommand{\bff}{{\boldsymbol f}}

\newtheorem{theorem}{Theorem}
\newtheorem{lema}{Lemma}

\newcounter{remark}
\def\theremark {\arabic{remark}}
\newenvironment{remark}{\refstepcounter{remark}\par\noindent{\bf Remark\ \theremark}\ }{\par}

\newtheorem{Proof}{Proof}

\newenvironment{proof}{\begin{Proof}\rm}{\hfill $\Box$ \end{Proof}}

\newcommand{\nrm}{\interleave}

\title{Error analysis of projection methods for non inf-sup stable mixed finite elements. The Navier-Stokes equations.}

\author{Javier de Frutos\thanks{Instituto de Investigaci\'on en Matem\'aticas (IMUVA),
Universidad de Valladolid, Spain.  Research supported
under grants MTM2013-42538-P (MINECO, ES) and MTM2016-78995-P (AEI/FEDER, UE). frutos@mac.uva.es} \and Bosco
Garc\'{\i}a-Archilla\thanks{Departamento de Matem\'atica Aplicada
II, Universidad de Sevilla, Sevilla, Spain. Research  under grant MTM2015-65608-P (MINECO, ES). bosco@esi.us.es}
  \and Julia Novo\thanks{Departamento de
Matem\'aticas, Universidad Aut\'onoma de Madrid, Instituto de
Ciencias Matem\'aticas CSIC-UAM-UC3M-UCM, Spain. Research supported
under grants MTM2013-42538-P (MINECO, ES) and MTM2016-78995-P (AEI/FEDER, UE). julia.novo@uam.es}}
\date{}

\begin{document}
\maketitle
\begin{abstract}
We obtain error bounds
 for  a modified Chorin-Teman (Euler non-incremental) method for non inf-sup stable mixed finite elements
applied to the evolutionary Navier-Stokes equations. The analysis of the classical Euler non-incremental  method is
 obtained as a particular case. We  prove that the modified Euler non-incremental
scheme has an inherent stabilization that allows the use of non inf-sup
stable mixed finite elements without any kind of extra added
stabilization. We show that it is also true in the case of the classical
Chorin-Temam method. The relation of
the methods with the so called pressure stabilized Petrov Galerkin method
(PSPG) is established. We do not assume non-local
compatibility conditions for the solution.
\end{abstract}


\section{Introduction}
 We analyze a modified Chorin-Teman (Euler non-incremental) projection method for non
inf-sup stable mixed finite elements with a pressure space $Q_h\subset H^1(\Omega)$.
As a particular case we obtain the analysis of the classical Chorin-Temam method. We prove that both the modified and the standard methods have an inherent stabilization of PSPG type
that allow the use of non inf-sup stable mixed finite elements without any kind of extra added stabilization. This result was known in the literature,
see for example \cite{Guermond_Quar_IJNMF}, but to our knowledge there were no available error bounds for the case of non inf-sup stable elements (see below for related results in~\cite{badia-codina}). In reference~\cite{Javier-Bosco-JuliaI} we considered  the case of the transient Stokes equations assuming enough regularity for the solution. In the present paper the analysis is applied to the evolutionary Navier-Stokes equations without assuming non-local compatibility conditions.  We consider a explicit treatment of the nonlinear convection term since this is easier to implement in practice, although our analysis, if slightly modified, also covers the case
of a fully implicit treatment of the nonlinear term. The analysis of the Chorin-Temam method holds under condition
$\Delta t \ge C h^2 $ (and assuming also $\Delta t=O(h)$)
which is in agreement with the
error bounds in \cite{badia-codina} where the authors prove error bounds for the Euler non-incremental scheme for LBB stable elements assuming also $\Delta t\ge C h^2$.  This result is also in agreement with the fact that had been observed in the literature that the standard Euler non-incremental scheme
provides computed pressures that behave unstably for $\Delta t$ small and fixed $h$ if non inf-sup
stable elements are used, see \cite{codina}.  With our error analysis we clarify this question since we show that, when~$\Delta t\rightarrow 0$,
the inherent PSPG stabilization of the method disappears. On the other hand, for the modified Euler non-incremental method that we propose, the PSPG stabilization does not disappear when $\Delta t \rightarrow 0$, which allows to use~$\Delta t$ as small as desired in this modified method.
Our results are
also in agreement with the classical results for the continuous in space Euler non-incremental method (see for example \cite{Guermond_overview}) since we prove that the rate of convergence in terms of $\Delta t$ in the $L^2$ norm of the velocity is one and the rate of convergence in the $H^1$ norm of the velocity and the $L^2$ norm of the pressure is one half.

It is well-known (see e.g., \cite[Corollary~2.1]{heyran1}) that the solution of the Navier-Stokes equations, no
matter how smooth the initial velocity and the forcing term are, cannot
be expected to have third spatial derivatives bounded up to $t=0$, unless certain nonlocal compatibility conditions (which
are difficult to check in practice and cannot be realistically assumed) are satisfied. For the pressure, the
same can be said for~second spatial derivatives.
To cope with this fact in this paper we obtain error bounds  that do not
require the above-mentioned compatibility conditions to be satisfied.

Of course, the Chorin-Temam projection method  is well known and this is not the first paper where the analysis of this method is considered. The analysis of the semidiscretization in time  is carried out in \cite{shen1}, \cite{shen2}, \cite{ranacher}, \cite{Prohl}, \cite{Prohl2}. In \cite{codina} the stability of the Chorin-Temam
projection method is considered and, in case of non inf-sup stable mixed finite elements, some a priori bounds
for the approximations to the velocity and pressure are obtained
but no error bounds are proven for this method. In  \cite{badia-codina} the Chorin-Teman method is considered together with both non inf-sup stable and inf-sup stable mixed finite elements.
In case of using non inf-sup stable mixed finite elements a local projection type stabilization is required in \cite{badia-codina} to get the
error bounds of the method. Both in the present paper and in~\cite{Javier-Bosco-JuliaI}, however, we get optimal error bounds without any extra stabilization for non inf-sup stable mixed finite elements.

For the Euler incremental scheme the analysis of the semidiscretization in time  can be found in~\cite{Prohl}.
The Euler incremental scheme with a spatial discretization based on inf-sup stable mixed finite elements is analyzed in \cite{Guermond_Quar2}. To our knowledge there is no error analysis for this method in case of using non-inf-sup stable elements other than the one in~\cite{Javier-Bosco-JuliaI}.  Some stability estimates can be found in \cite{codina} for the method with added stabilization terms more related to local projection stabilization than to the PSPG stabilization we consider in the present paper.
A stabilized version of  the incremental scheme is also proposed in \cite{minev} although no error bounds are proved.
Finally, for an overview on projection methods we refer the reader to \cite{Guermond_overview}.

Being the Chorin-Temam projection method an old one, it has seen the appearance of many alternative methods during the years, many of which possess better convergence properties. The purpose of the present
paper is not to discuss its advantages of disadvantages with respect to newer methods, but just to analyze the method when used in
combination with non inf-sup stable elements, a task not fully carried out in the previous literature.

The outline of the paper is as follows. In the first section we introduce some notation. In Section 3 we state some results about a stabilized Stokes approximation
that was introduced in \cite{Javier-Bosco-JuliaI}. In Section 4 we get the error analysis of the method for the transient Stokes equations.
Finally, in the last section we prove the error bounds for the method for the Navier-Stokes equations.  The analysis is based on a stability plus consistency arguments with
stability restricted to $h$-dependent thresholds and is strongly based on the results for the transient Stokes equations obtained in Section 4.

\section{Preliminaries and notation}
Throughout the paper, standard notation is used for Sobolev
spaces and corresponding norms. In particular, given a measurable set $\omega\subset\mathbb{R}^d$, $d=2,3$, its Lebesgue
measure is denoted by~$|\omega|$,  the inner product in $L^2(\omega)$ or $L^2(\omega)^d$ is
denoted by $(\cdot,\cdot)_\omega$ and the notation $(\cdot,\cdot)$ is used instead
of $(\cdot,\cdot)_\Omega$.  The semi norm in $W^{m,p}(\omega)$ will be denoted
by $|\cdot|_{m,p,\omega}$ and, following~\cite{Constantin-Foias}, we define the norm~$\left\|\cdot\right\|_{m,p,\omega}$ as
$$
\left\| f\right\|_{m,p,\omega}^p=\sum_{j=0}^m \left|\omega\right|^{\frac{p(j-m)}{d}} \left| f\right|_{j,p,\omega}^p,
$$
so that $\left\|f\right\|_{m,p,\omega}\left|\omega\right|^{\frac{m}{d}-\frac{1}{p}}$ is scale invariant. We will also use
the conventions $\|\cdot\|_{m,\omega}=\|\cdot\|_{m,2,\omega}$ and
$\|\cdot\|_{m}=\|\cdot\|_{m,2,\Omega}$. As it is usual we will use the special notation $H^s(\omega)$ to denote $W^{s,2}(\omega)$ and
we will denote by $H_0^1(\Omega)$ the subspace of functions of  $H^1(\Omega)$  satisfying homogeneous Dirichlet boundary conditions. Finally, $L^2_0(\Omega)$ will denote the subspace of function of $L^2(\Omega)$ with zero mean.

Let us denote by $\mathcal T_h$ a triangulation of the domain $\Omega$, which, for simplicity, is assumed to have a Lipschitz polygonal boundary.
On $\mathcal T_h$, we consider the  finite element spaces $V_h \subset V=H_0^1(\Omega)^d$ and
$Q_h \subset L_0^2(\Omega) \cap H^1(\Omega)$ based on local polynomials of degree $k$ and
$l$ respectively. Equal degree polynomials for velocity and pressure are allowed. In the sequel it will be assumed that the family of meshes are regular.

Concerning the discretization, we shall  assume that the family of
meshes is quasi-uniform that is for a constant~$\Lambda\ge 1$, the following
inequality holds
\begin{equation}
\label{eq:quasi}
h/h_K \le \Lambda, \qquad \forall K\in {\mathcal T}_h,
\end{equation}
where $h_K$ is the diameter of the $K \in {\mathcal T}_h$ and $h=\max_{K \in {\mathcal T}_h}h_K$.

We shall also assume that the triangulations are regular enough so that
for a constant $c_{\mathrm{inv}}>0$ the following
inequality holds for each $v_{h} \in V_{h}$, see e.g., \cite[Theorem 3.2.6]{Cia78},
\begin{equation}
\label{inv} \| \bv_{h} \|_{W^{m,p}(K)} \leq c_{\mathrm{inv}}
h_K^{l-m-d\left(\frac{1}{q}-\frac{1}{p}\right)}
\|\bv_{h}\|_{W^{l,q}(K)},
\end{equation}
where $0\leq l \leq m \leq 1$, $1\leq q \leq p \leq \infty$, and $h_K$
is the size (diameter) of the mesh cell $K \in \mathcal T_h$.

We will denote by $I_h \bu \in V_{h}$ the Lagrange interpolant of a continuous function $\bu$. The following bound can be found  in
\cite[Theorem 3.1.6]{Cia78}
\begin{equation}\label{cota_inter}
|\bu-I_h \bu|_{W^{m,p}(K)^d}\le c_\text{\rm int} h^{l'-m-d\left(\frac{1}{q}-\frac{1}{p}\right)}|\bu|_{W^{l',q}(K)^d},\quad 0\le m\le 1\le l',
\end{equation}
where $l'>d/q$ when $1< q\le \infty$ and $l'\ge d$ when $q=1$.

Let~$\lambda$ be the smallest eigenvalue of  $A=-\Delta$ subject to homogeneous Dirichlet  boundary conditions, $\Delta$ being the Laplacian operator in~$\Omega$. Then, it is well-known that there exists a scale-invariant positive constant~$c_{-1}$ such that
\begin{equation}
\label{eq:cota_menos1}
\left\| \bv\right\|_{-1}\le c_{-1}\lambda^{-1/2} \left\| \bv\right\|_0,\qquad
\bv\in L^{2}(\Omega)^d,
\end{equation}
and, also,
\begin{eqnarray*}
\left\| \bv\right\|_{0}\le \lambda^{-1/2} \left\|\nabla \bv\right\|_0,\qquad
\bv\in H^1_0(\Omega)^d,
\end{eqnarray*}
this last inequality also known as the Poincar\'e inequality. As a consequence of the above, there exist a scale-invariant constant
$c_P>0$ such that
\begin{equation}
\label{eq:poincare2}
\left\| \bv\right\|_{1}\le c_P \left\|\nabla \bv\right\|_0,\qquad
\bv\in H^1_0(\Omega)^d,
\end{equation}

We will use the following well-known inequalities:
\begin{description}
\item{i)} Sobolev's inequality,
\cite{Adams}: For $s>0$
there exist a scale-invariant constant $c_s>$ such that for~$p\in[1,\infty)$
satisfying $\frac{1}{p}\ge \frac{1}{2}-\frac{s}{d}$, the following inequality
holds
\begin{equation}\label{eq:sob}
\|v\|_{L^{p}(\Omega)} \le c_{s}\left|\Omega\right|^{\frac{s}{d}-\frac{1}{2}+
\frac{1}{p}} \| v\|_{s}, \qquad
\quad v \in
H^{s}(\Omega).
\end{equation}
For $p=\infty$, the relation is valid if $0>\frac{1}{2}-\frac{s}{d}$.

\item{ii)} Agmon's inequality,
\begin{equation}
\label{eq:agmon}
\left\|v \right\|_\infty\le c_{\mathrm{A}}\left\{
\begin{array}{lcl}  \left\|v\right\|_0^{1/2} \left\|v\right\|_2^{1/2},&\quad
d=2,\\
\left\|v\right\|_1^{1/2} \left\|v\right\|_2^{1/2},&\quad
d=3,\end{array}\right.
\qquad v\in H^2(\Omega).
\end{equation}
The case $d=2$ is a direct consequence of~~\cite[Theorem 3.9]{Agmon}.
For $d=3$, a proof for domains of class~$C^2$ can be found in~\cite[Lemma~4.10]{Constantin-Foias}, but thanks to the Calder\'on extension  theorem~(see e.g., \cite[Theorem 4.32]{Adams} the proof is valid for bounded Lipschitz domains.

\item{iii)} The following version of H\"older's inequality
\begin{equation}
\label{eq:tri_holder}
\Bigl|\int_{\Omega}v_1v_2v_3\,dx\Bigr| \le
\left\|v_1\right\|_{L^{p_1}(\Omega)} \left\|v_2\right\|_{L^{p_2}(\Omega)} \left\|v_3\right\|_{L^{p_3}(\Omega)}
,\quad  \frac{1}{p_1}+
 \frac{1}{p_2}+ \frac{1}{p_3}=1.
\end{equation}
We shall frequently apply this inequality with~$p_1=2$, $p_2=2d/(d-1)$ and
$p_3=2d$, or~$p_1=\infty$, and $p_2=p_3=2$.

\item{iv}) The following inequality
\begin{equation}
\label{eq:parti_ineq}
\left\|v\right\|_{L^{\frac{2d}{d-1}}(\Omega)} \le c_{1}^{1/2}\left\|v\right\|_0^{1/2}
\left\|\nabla  v\right\|_0^{1/2},\qquad v\in H^1(\Omega).
\end{equation}
which is a consequence of Sobolev's inequality and the convexity
inequality (see e. g.,
\cite[\S~II.1]{Galdi}).
\end{description}
All previous inequalities are also valid for vector-valued functions.
\section{A Stabilized Stokes approximation}

Let us consider the Stokes problem
\begin{eqnarray}\label{eq:stokes}
-\nu\Delta {\mbf s}+\nabla z&=&\hat \bg,\quad {\rm in}\quad \Omega\nonumber\\
\nabla \cdot {\mbf s}&=& 0,\quad {\rm in}\quad \Omega\\
{\mbf s}&=&{\bs 0},\quad{\rm in}\quad \partial \Omega.\nonumber
\end{eqnarray}
As in \cite{Javier-Bosco-JuliaI} we define the stabilized Stokes approximation to (\ref{eq:stokes}) as the mixed
finite element approximation $({\mbf s}_h,z_h)\in (V_h,Q_h)$ satisfying
\begin{eqnarray}\label{eq:pro_stokes}
\nu(\nabla {\mbf s}_h,\nabla {\bs\chi}_h)+(\nabla z_h,{\bs\chi}_h)&=&(\hat\bg,{\bs\chi}_h),\quad \forall {\bs\chi}_h\in V_h,\\
(\nabla \cdot {\mbf s}_h,\psi_h)&=&-\delta(\nabla z_h,\nabla \psi_h),\quad \forall \psi_h\in Q_h,\label{eq:pro_stokes2}
\end{eqnarray}
where $\delta$ is a constant parameter.
Observe that from~\eqref{eq:stokes} and~\eqref{eq:pro_stokes} it follows that the errors
${\mbf s}_h-{\mbf s}$ and~$z_h-z$ satisfy that
\begin{equation}
\label{eq:gal_orthog}
\nu(\nabla({\mbf s}_h-{\mbf s}),\nabla{\bs\chi}_h) + (\nabla (z_h-z),{\bs\chi}_h)=0,\qquad \forall{\bs\chi}_h\in V_h.
\end{equation}
From now on we will use $C$ to denote a generic non-dimensional constant.

We now state two lemmas that will be used in the sequel.
The proof of the following lemma can be found in~\cite[Lemma 3]{Burman}, see also \cite[Lemma 2.1]{John_Novo_PSPG}.
\begin{lema}\label{lema_presion}
For  $\psi_h \in Q_h $ it holds
\begin{eqnarray*}
\|\psi_h\|_0\le C h\|\nabla \psi_h\|_0+C\sup_{{\bs\chi}_h\in V_h}\frac{(\psi_h,\nabla \cdot {\bs\chi}_h)}{\|{\bs\chi}_h\|_1}.
\end{eqnarray*}
\end{lema}
\begin{lema}\label{le:duality} There exist a constant~$C>0$ such that
for any $\bv\in H^1_0(\Omega)^d$ with $\hbox{\rm div}(\bv)=0$, $q\in~L^2_0(\Omega)$,  $\bv_h\in V_h$ and~$q_h\in Q_h$ satisfying
\begin{eqnarray}
\label{eq:gal_orthogv}
\nu(\nabla(\bv_h-\bv),\nabla{\bs\chi}_h) + (\nabla (q_h-q),{\bs\chi}_h)=0,\qquad \forall{\bs\chi}_h\in V_h,\\
\label{eq:gal_orthogv2}
(\nabla\cdot(\bv_h-\bv),\psi_h)+\delta(\nabla q_h,\nabla\psi_h)=0,\qquad
\forall \psi_h\in Q_h,
\end{eqnarray}
the following bounds hold:
\begin{equation}
\label{eq:cota1}
\nu^{1/2}\left\|\nabla \bv_h\right\|_0+\delta^{1/2} \left\|\nabla q_h\right\|_0
\le C\bigl(\nu^{1/2}\left\|\nabla \bv\right\|_0 +\nu^{-1/2}\left\|q\right\|_0\bigr),
\end{equation}
\begin{equation}
\left\| \bv_h-\bv\right\|_0 \le C\left(h\left(\left\|\nabla(\bv-\bv_h)\right\|_0+\nu^{-1}\left\|q-q_h\right\|_0\right)+\delta\left\|\nabla q_h\right\|_0\right).
\label{eq:cota2}
\end{equation}
\end{lema}
\begin{proof}
Observe that since $\hbox{\rm div}(\bv)=0$, relation~(\ref{eq:gal_orthogv2}) can be written in the form
$(\nabla\cdot \bv_h,\psi_h)+\delta(\nabla q_h,\nabla\psi_h)=0$,
for all $\psi_h\in Q_h$. Then taking~$\psi_h=q_h$ in this relation,
${\bs\chi}_h=\bv_h$ in~(\ref{eq:gal_orthogv}) and summing both equations, the bound~(\ref{eq:cota1}) easily follows.
The proof of (\ref{eq:cota2}) can be found in \cite[Lemma~2]{Javier-Bosco-JuliaI}.
\end{proof}

In the sequel we will assume
\begin{equation}\label{eq:cond_delta}
\frac{1}{\nu\rho_1^2}h^2\le \delta,
\end{equation}
for a positive constant $\rho_1$.
The following bounds hold for the stabilized Stokes approximation solving (\ref{eq:pro_stokes})-(\ref{eq:pro_stokes2}) assuming condition (\ref{eq:cond_delta}) holds, see \cite{Javier-Bosco-JuliaI}.
\begin{equation}\label{eq:error_steady_simplified_linear}
\begin{split}
\nu^{1/2}\|\nabla({\mbf s}-{\mbf s}_h)\|_0+\delta^{1/2}&\|\nabla(z-z_h)\|_0\\
&\le C \frac{h}{\nu^{1/2}}(\nu \|{\mbf s}\|_{2}+\|z\|_1)+C \delta^{1/2}\|z\|_1,\\
\|z-z_h\|_0&\le C h(\nu\|{\mbf s}\|_{2}+\|z\|_1)+C(\nu\delta)^{1/2}\|z\|_1\\
 \|{\mbf s}-{\mbf s}_h\|_0&\le C\frac{h^{2}}{\nu}(\nu\|{\mbf s}\|_{2}+\|z\|_{1})+C\delta \|z\|_1.
\end{split}
\end{equation}
\subsection{A priori bounds for the stabilized Stokes approximation}
We will get some a priori bounds for the stabilized Stokes approximation that will be needed in the sequel.
They are a consequence of~Lemma~\ref{le:duality}. In fact applying this result with
 $\bv_h={\mbf s}_h$ and~$q_h=z_h$,
(\ref{eq:cota1})  implies that
\begin{equation}\label{eq:cotas1}
\nu \|\nabla {\mbf s_h}\|_0^2+\delta\|\nabla z_h\|_0^2\le C
\bigl(\nu \|\nabla {\mbf s}\|_0^2+\nu^{-1}\|z\|_0^2\bigr).
\end{equation}

Similarly, from~(\ref{eq:cota2}) it follows that
\begin{equation}
\label{eq:una_cota_sh-u}
\begin{split}
\left\| {\mbf s}_h-{\mbf s}\right\|_0 &\le  C\left(h\left(\left\|\nabla({\mbf s}_h-{\mbf s})\right\|_0+\nu^{-1}\left\|z-z_h\right\|_0\right)+\delta\left\|\nabla z_h\right\|_0\right)
\\
 &\le   Ch\left(\left\|\nabla{\mbf s}_h\right\|_0+\left\|\nabla {\mbf s}\right\|_0+\nu^{-1}\left(\left\|z\right\|_0+\left\|z_h\right\|_0\right)\right)
 +C\delta\left\|\nabla z_h\right\|_0\\
 &\le
 Ch\left(\left\|\nabla {\mbf s}\right\|_0+\nu^{-1}\left(\left\|z\right\|_0+\left\|z_h\right\|_0\right)\right )
 +C\delta^{1/2}\left(\nu^{1/2}\left\|\nabla{\mbf s}\right\|_0+\nu^{-1/2}\left\|z\right\|_0\right),
\end{split}
\end{equation}
where in the last inequality we have applied~(\ref{eq:cotas1}). Now observe
that from~Lemma~\ref{lema_presion} and~(\ref{eq:gal_orthog}) it follows that
\begin{align*}
\|z_h\|_0&\le C h \delta^{-1/2} \delta^{1/2}\|\nabla z_h\|_0+\sup_{{\bs\chi}_h\in V_h}\frac{(z_h,\nabla \cdot {\bs\chi}_h)}{\|{\bs\chi}_h\|_1}\nonumber\\
&\le C h \delta^{-1/2} \delta^{1/2}\|\nabla z_h\|_0+\left\|z\right\|_0+
\nu\left\|\nabla({\mbf s} -{\mbf s}_h)\right\|_0.
\end{align*}
Recalling~(\ref{eq:cond_delta}) and applying (\ref{eq:cotas1}) we have
\begin{equation}
\label{eq:cota_z_h_0}
\|z_h\|_0 \le C \bigl( \nu\left\|\nabla {\mbf s}\right\|_0+\| z\|_0\bigr).
\end{equation}
To bound $\|\nabla z_h\|_0$ we add and subtract $\nabla z$ and apply (\ref{eq:error_steady_simplified_linear}) and (\ref{eq:cond_delta}) to obtain
\begin{eqnarray}\label{eq:cota_nabla_z_h}
\|\nabla z_h\|_0\le \|\nabla (z_h-z)\|_0+\|\nabla z\|_0&\le& \delta^{-1/2}C\frac{h}{\nu^{1/2}}(\nu\|{\mbf s}\|_2+\|z\|_1)+\|z\|_1\nonumber\\
&\le& C(\nu\|{\mbf s}\|_2+\|z\|_1).
\end{eqnarray}
From~(\ref{eq:una_cota_sh-u}) and (\ref{eq:cota_z_h_0}) we get
\begin{equation}\label{eq:cotas0}
\|{\mbf s}_h\|_0\le Ch\left(\left\|\nabla{\mbf s}\right\|_0+\nu^{-1}\left\|z\right\|_0\right)+C\left\|{\mbf s}\right\|_0.
\end{equation}
Finally, we will also use the following bound
\begin{eqnarray}\label{complemento24}
\|{\mbf s}-{\mbf s}_h\|_0\le C ({h}+\nu^{1/2}\delta^{1/2})\left(\|{\mbf s}\|_1+\nu^{-{1}}\|z\|_0\right).
\end{eqnarray}
To prove (\ref{complemento24}) we first observe that taking into account (\ref{eq:cotas1}) and (\ref{eq:cota_z_h_0})
and adding and subtracting ${\mbf s}$ and $z$ respectively we  get
\begin{eqnarray}\label{eq:la_continua}
\nu^{1/2}\|\nabla ({\mbf s}-{\mbf s}_h)\|_0&\le& C(\nu^{1/2} \|{\mbf s}\|_1+\nu^{-1/2}\|z\|_0),\\
\|z_h-z\|_0&\le& C(\nu\|{\mbf s}\|_1+\|z\|_0).\nonumber
\end{eqnarray}
Applying the bound (\ref{eq:cota2}) to $(\bv_h,q_h)=({\mbf s}_h,z_h)$ together with (\ref{eq:la_continua}) and (\ref{eq:cotas1})   we reach (\ref{complemento24}).

\section{Transient Stokes equations}
We now consider the evolutionary Stokes equations
\begin{eqnarray}\label{eq:evo_stokes}
\bv_t-\nu\Delta {\bv}+\nabla q&=&{\mbf g},\qquad \quad{\rm in}\quad \Omega\nonumber\\
\nabla \cdot {\bv}&=& 0,\qquad\quad {\rm in}\quad \Omega\\
{\bv}&=&{\bs 0},\qquad\quad\! {\rm on}\quad\partial \Omega,\nonumber\\
\bv(0,\bx)&=&\bv_0(\bx),\quad  {\rm in}\quad \Omega.\nonumber
\end{eqnarray}
We shall assume that there are positive constants $M_1$ and~$M_2$
such that for $t\in[0,T]$,
\begin{align}\label{eq:M1-M2}
\left\|\bv(t)\right\|_1+\nu^{-1}\left\|q(t)\right\|_0\le M_1,\quad
\left\|\bv(t)\right\|_2+\nu^{-1}\left(\left\|q(t)\right\|_1+\left\|\bv_t(t)\right\|_0\right)\le M_2,
\end{align}
and, following the analysis in~\cite{heyran1}, for $k\ge 2$ integer, we shall
assume that the following quantities are finite
\begin{align}
M_{k,1}&=\max_{0\le t\le T}(t/T)^{k/2-1}\bigl(
\left\| \bv(t)\right\|_{k}+\nu^{-1}\left\| q(t)\right\|_{H^{k-1}/{\mathbb R}}\bigr),\label{eq:u-infa}\\
M_{k,2}&=\max_{0\le t\le T}(t/T)^{k/2-1}\bigl(\nu^{-1}\left\| \bv_t(t)\right\|_{k-2}+\nu^{-2}\left\| q_t(t)\right\|_{H^{k-3}/{\mathbb R}}\bigr), \label{eq:u-inf}\\
K_{k,2}^2&=\int_0^T
\Bigl(\frac{t}{T}\Bigr)^{k-3}
\bigr(\nu^{-2}\left\| \bv_t(t)\right\|_{k-2}^2+
\nu^{-4}\left\| q_t(t)\right\|_{H^{k-3}/{\mathbb R}} ^2\bigl)\,dt,
\label{eq:u-int}
\end{align}
together with,
\begin{equation}
\label{eq:cota_vtt}
K_{4,3}^2=\nu^{-4}\int_0^T\frac{t}{T}\left\| \bv_{tt}\right\|_0^2\,dt,
\end{equation}
and
\begin{equation}
\hat K_3^2=\nu^{-2}\int_0^T \left\| \bg_t\right\|_0^2 \,dt.
\label{eq:cota_gt}
\end{equation}

We consider the modified Euler non-incremental scheme that has been introduced in \cite{Javier-Bosco-JuliaI}.
We will denote by $(\bv_h^n,\tilde \bv_h^n,q_h^n)$, $n=1,2,\ldots,$\  $\tilde \bv_h^n\in V_h$, $q_h^n\in Q_h$ and $ \bv_h^n\in V_h+\nabla Q_h$
the approximations to the velocity and pressure at time $t_n=n\Delta t$, $\Delta t=T/N$, $N>0$ obtained with the modified Euler non-incremental scheme
\begin{eqnarray}\label{eq:eu_non}
&&\left(\frac{\tilde \bv_h^{n+1}-\bv_h^n}{\Delta t},{\bs\chi}_h\right)+\nu(\nabla \tilde \bv_h^{n+1},\nabla {\bs\chi}_h)=({\mbf g}^{n+1},{\bs\chi}_h),\quad \forall {\bs\chi}_h\in V_h\nonumber\\
&&(\nabla \cdot \tilde \bv_h^{n+1},\psi_h)=-\delta(\nabla q_h^{n+1},\nabla \psi_h),\quad \forall \psi_h\in Q_h,\\
&&\bv^{n+1}_h=\tilde \bv_h^{n+1}-\delta\nabla q_h^{n+1}.\nonumber
\end{eqnarray}
Let us observe that for $\delta=\Delta t$ in (\ref{eq:eu_non}) we have the classical
 Chorin-Temam (Euler non-incremental) scheme \cite{Chorin}, \cite{Temam}. In case $\delta=\Delta t$ we can remove $\bv_h^n$ from
 (\ref{eq:eu_non}) inserting the expression of $\bv_h^n$ from the last equation in (\ref{eq:eu_non}) into the first equation to get
\begin{eqnarray}\label{eq:eu_non_tilde}
&&\left(\frac{\tilde \bv_h^{n+1}-\tilde \bv_h^n}{\Delta t},{\bs\chi}_h\right)+\nu(\nabla \tilde \bv_h^{n+1},\nabla {\bs\chi}_h)+(\nabla q_h^n,{\bs \chi}_h)=({\mbf g}^{n+1},{\bs\chi}_h),\quad \forall {\bs\chi}_h\in V_h\qquad\\
&&(\nabla \cdot \tilde \bv_h^{n+1},\psi_h)=-\delta(\nabla q_h^{n+1},\nabla \psi_h),\quad \forall \psi_h\in Q_h.\label{eq:eu_non_tilde2}
\end{eqnarray}
The method we propose is (\ref{eq:eu_non_tilde}) for $\delta$, in general, different from $\Delta t$.
We suggest to take $\delta$ satisfying (\ref{eq:cond_delta}).
As a consequence of the error analysis of this section we will get the error bounds for the classical
Euler non-incremental scheme assuming in that case $\delta=\Delta t$.

To get the error bounds of the method we compare the approximation $(\tilde \bv_h^n,q_h^n)$ defined in (\ref{eq:eu_non_tilde})-(\ref{eq:eu_non_tilde2})  with the stabilized Stokes approximation defined in the previous section. More precisely, let us denote by
$({\mbf s}_h^n,z_h^n)=({\mbf s}_h(t_n),z_h(t_n))\in V_h\times Q_h$ the stabilized Stokes approximation of the solution $(\bv,p)$ of (\ref{eq:evo_stokes}) at time $t_n$ satisfying
\begin{eqnarray}\label{eq:pro_stokes_evo}
\nu(\nabla {\mbf s_h},{\bs\chi}_h)+(\nabla z_h,{\bs\chi}_h)&=&(\hat{\mbf g},{\bs\chi}_h),\quad {\bs\chi}_h\in V_h,\\
(\nabla \cdot {\mbf s_h},\psi_h)&=&-\delta(\nabla z_h,\nabla \psi_h),\quad \forall \psi_h\in Q_h\nonumber,
\end{eqnarray}
where $\hat{\mbf  g}={\mbf g}-\bv_t$.
Let us observe that the error bounds of Section~3.1 hold with $(\mbf{s},z)=(\bv,q)$.
Taking time derivatives  in (\ref{eq:cotas1}) and (\ref{eq:cotas0}) we also reach
\begin{align}\label{eq:cota0deri1}
\|({\mbf s}_h)_t\|_0&\le  Ch\left(\left\|\nabla{\mbf s}_t\right\|_0+\nu^{-1}\left\|z_t\right\|_0\right)+C\left\|{\mbf s}_t\right\|_0.
\end{align}
In the sequel we will denote by
$$
\tilde \be_h^n=\tilde \bv_h^n-{\mbf s}_h^n,\quad  r_h^n=q_h^n-z_h^n.
$$
From  (\ref{eq:eu_non_tilde}), (\ref{eq:eu_non_tilde2}) and (\ref{eq:pro_stokes_evo}) one obtains the following error equation for all ${\bs\chi}_h\in V_h$, $\psi_h\in Q_h$
\begin{align}
\Bigl(\frac{\tilde \be_h^{n+1}-\tilde \be_h^n}{\Delta t},{\bs\chi}_h\Bigr)
+\nu (\nabla \tilde \be_h^{n+1},\nabla {\bs\chi}_h)+&(\nabla r_h^n,{\bs\chi}_h)=\nonumber\\
&(\tau_h^n,{\bs\chi}_h)-(\nabla (z^n_h-z_h^{n+1}),{\bs\chi}_h),\nonumber\\
(\nabla \cdot \tilde \be_h^{n+1},\psi_h)+&\delta(\nabla r_h^{n+1},\nabla \psi_h)=0.
\label{eq:er3_orig}
\end{align}
where
\begin{equation}
\label{eq:tau_h}
\tau_h^n=\bv_t^{n+1}-\frac{{\mbf s}_h^{n+1}-{\mbf s}_h^n}{\Delta t}=(\bv_t^{n+1}-({\mbf s_h})_t^{n+1})+\left(({\mbf s_h})_t^{n+1}-\frac{{\mbf s}_h^{n+1}-{\mbf s}_h^n}{\Delta t}\right).
\end{equation}

To estimate the errors $\tilde\be_h^n$ and~$r_h^n$ we will use the following stability result.
\begin{lema}\label{lema_stab_evol} Let $(\bw_h^n)_{n=0}^\infty$
and~$(\bb_h^n)_{n=0}^\infty$ sequences in~$V_h$
and~$(y_h^n)_{n=0}^\infty$ and~$(d_h^n)_{n=0}^\infty$ sequences in
$Q_h$ satisfying for all ${\bs\chi}_h\in V_h$ and $\psi_h\in Q_h$
\begin{align*}
\left(\frac{\bw_h^{n+1}-\bw_h^n}{\Delta t},{\bs\chi}_h\right)+\nu (\nabla \bw_h^{n+1},\nabla {\bs\chi}_h)+(\nabla y_h^n,{\bs\chi}_h)=&(\bb_h^n+\nabla d_h^n,{\bs\chi}_h),\\
(\nabla \cdot \bw_h^{n+1},\psi_h)+\delta(\nabla y_h^{n+1},\nabla \psi_h)=&0.{}\quad
\end{align*}
Assume condition
\begin{equation}\label{eq:cond_delta2}
 \Delta t \le \delta
\end{equation}
holds.
Then, for $0\le n_0\le n-1$ there exits a non-dimensional constant $c_0$ such that the following bounds hold
\begin{equation}\label{eq:stab_evol_1}
\begin{split}
\|\bw_h^{n}\|_0^2+\sum_{j=n_0}^{n-1}\|\bw_h^{j+1}&-\bw_h^j\|_0^2
+\Delta t\sum_{j=n_0}^{n-1}\bigl(\nu\|\nabla\bw_h^{j+1}\|_0^2+{\delta}\|\nabla y_h^{j+1}\|_0^2\bigr)\\
&\le c_0\Bigl(\|\bw_h^{n_0}\|_0^2+ \Delta t\sum_{j=n_0}^{n-1}\bigl(\nu^{-1}\|\bb_h^j\|_{-1}^2+\delta\|\nabla d_h^j\|_0^2\bigr)\Bigr).
\end{split}
\end{equation}
\begin{align}\label{eq:stab_evol_1t}
t_n\|&\bw_h^{n}\|_0^2+\sum_{j=n_0}^{n-1}{t_{j+1}}\|\bw_h^{j+1}-\bw_h^j\|_0^2
+\Delta t\sum_{j=n_0}^{n-1}t_{j+1}\bigl(\nu\|\nabla\bw_h^{j+1}\|_0^2+{\delta}\|\nabla y_h^{j+1}\|_0^2\bigr)
\nonumber\\
&\le c_0\Bigl(t_{n_0}\|\bw_h^{n_0}\|_0^2+\Delta t\sum_{j=n_0}^{n}\|\bw_h^{j}\|_0^2
+ \Delta t\sum_{j=n_0}^{n-1}
t_{j+1}\bigl(t_{j+1}\|\bb_h^j\|_{0}^2+\delta\|\nabla d_h^j\|_0^2\bigr)\Bigr).
\end{align}
\begin{equation}
\label{eq:stab_evol_2}
\begin{split}
\sum_{j=n_0}^{n-1}{\Delta t}&\biggl\|\frac{\bw_h^{j+1}-\bw_h^{j}}{\Delta t}\biggr\|_0^2+\nu\|\nabla \bw_h^{n}\|_0^2+ \delta\|\nabla y_h^{n}\|_0^2+\nu\sum_{j=n_0}^{n-1}\|\nabla(\bw_h^{j+1}-\bw_h^j)\|_0^2\\
&{}\quad\le
c_0\Bigl(\nu\|\nabla \bw_h^{n_0}\|_0^2+\delta\|\nabla y_h^{n_0}\|_0^2
+ \Delta t\sum_{j=n_0}^{n-1}\bigl(\|\bb_h^j\|_0^2+\|\nabla d_h^j\|_0^2\bigr)\Bigr)
\end{split}
\end{equation}
and
\begin{equation}\label{unamas}
\begin{split}
\sum_{j=n_0}^{n-1}&{t_{j+1}}\Delta t \biggl\|\frac{\bw_h^{j+1}-\bw_h^j}{\Delta t}\biggr\|_0^2+\nu t_n\|\nabla \bw_h^n\|_0^2+\delta t_n\|\nabla y_h^n\|_0^2\\
&\le c_0\Bigl(\Delta t\sum_{j=n_0}^{n-1}t_{j+1}\bigl(\|b_h^j\|_0^2+\|\nabla d_h^j\|_0^2\bigr)+\Delta t\sum_{j=n_0}^{n-1}\bigl(\nu\|\nabla \bw_h^j\|_0^2+\delta\|\nabla y_h^j\|_0^2\bigr)\Bigr).
\end{split}
\end{equation}
\end{lema}
\begin{proof}
The proof of (\ref{eq:stab_evol_1}), (\ref{eq:stab_evol_1t}) and (\ref{eq:stab_evol_2}) can be found in \cite[Lemma 3]{Javier-Bosco-JuliaI}.
The proof of (\ref{unamas}) can be easily reached arguing as in the proof of (\ref{eq:stab_evol_2}).
\end{proof}

\begin{remark} As commented in~\cite{Javier-Bosco-JuliaI}, it is possible to change condition~\eqref{eq:cond_delta2} by $\Delta t< 2\delta$, but this requires a more elaborate proof than that presented in~\cite{Javier-Bosco-JuliaI}.
\end{remark}

In the sequel, although it is not strictly necessary to prove our results, we will assume that
\begin{equation}
\label{delta<T}
\delta \le T.
\end{equation}
to simplify some of the expressions below.

\begin{theorem}\label{Th1} Let $(\bv,q)$ be the solution of (\ref{eq:evo_stokes}) and let $(\tilde \bv_h^n,q_h^n)$, $n\ge 1$,  be the solution of
(\ref{eq:eu_non_tilde})-(\ref{eq:eu_non_tilde2}). Assume $\delta$ satisfies condition (\ref{eq:cond_delta}) and~(\ref{delta<T}), and
that $\Delta t$ satisfies condition (\ref{eq:cond_delta2}).
Then, the following bounds hold
\begin{eqnarray}\label{cota_th1_l2}
t_n\|\tilde \bv_h^n-\bv(t_n)\|_0^2 &\le&  Ct_n\left( \|\tilde \bv_h^0-\bv(0)\|_0^2+{\Delta t}^2\left\|\nabla  r_h^0\right\|_0^2\right)\nonumber\\
&& \quad+C_1 t_n {\Delta t}^2+C_2 t_n (h^4+\delta^2\nu^2),
\end{eqnarray}
where
$C_1$ and $C_2$ are defined as
\begin{eqnarray}
C_1&=&C\nu^2\left(C_2+\nu^2K_{4,3}^2T+\hat K_{3}^2T+M_{2,2}^2\right),\label{laC1}\\
C_2&=&C\left(\nu^2K_{4,2}^2T+\nu K_{3,2}^2+M_{2,1}^2\right)\label{laC2},
\end{eqnarray}
Moreover, it also holds,
\begin{equation}\label{cota_th1_H1}
\begin{split}
\Delta t\sum_{j=1}^{n}\bigl(\nu\|\nabla (\tilde \bv_h^j-\bv(t_j))\|_0^2&+\delta \|\nabla (q_h^j-q(t_j))\|_0^2\bigr)
\\
&\le C \|\tilde\bv_h^0-\bv(0)\|_0^2 +\tilde C_1 \Delta t+\tilde C_2(h^2+\nu\delta).
\end{split}
\end{equation}
where, assuming (\ref{delta<T}),
\begin{eqnarray}\label{laCtilde1}
\tilde C_1&=&C_1((\nu\lambda)^{-1}+T),\\
\tilde C_2&=&C_2(\lambda^{-1}+ \hbox{\rm diam}(\Omega)^2 + \nu T).
\label{laCtilde2}
\end{eqnarray}
\end{theorem}
\begin{proof}
In view of (\ref{eq:er3_orig}) we can apply~(\ref{eq:stab_evol_1t}) for
$\bw_h^n=\tilde\be_h^n$, $y_h^n= r_h^n$, $\bb_h^n=\tau_h^n$
and~$d_h^n=z_h^{n+1}-z_h^n$. It follows that
\begin{equation}
\label{eq:auxa1}
\begin{split}
t_n\|&\tilde \be_h^{n}\|_0^2
+\Delta t\sum_{j=1}^{n}t_{j}\bigl(\nu\|\nabla\tilde \be_h^{j}\|_0^2+{\delta}\|\nabla r_h^{j}\|_0^2\bigr)
\\
& \le c_0\Bigl(\Delta t\sum_{j=0}^{n}
\|\tilde\be_h^{j}\|_0^2+\sum_{j=0}^{n-1}\Delta t\bigl(t_{j+1}^2\|\tau_h^j\|_{0}^2+\delta t_{j+1}\|\nabla (z_h^{j+1}-z_h^j)\|_0^2\bigr)\Bigr).
\end{split}
\end{equation}
To bound the second term on the right-hand side of~(\ref{eq:auxa1}),
we notice that
$t_{j+1}/t_j\le 2$ for~$j=1,\ldots,n-1$, so that we may
write
$$
\Delta t\sum_{j=0}^{n-1} t_{j+1}^2\|\tau_h^j\|_{0}^2
\le C t_n\Delta t
\sum_{j=0}^{n-1}t'_{j}\|\tau_h^j\|_{0}^2,
$$
where $t'_j=\max(\Delta t, t_j)$.
From  definition (\ref{eq:tau_h}) we may write
\begin{align}\label{tercero1}
\|\tau_h^j\|_0^2\le 2 \left\|\bv_t^{j+1}-\frac{\bv^{j+1}-\bv^j}{\Delta t}\right\|_0^2+\frac{2}{{\Delta t}^2}\left\|(\bv^{j+1}-{\mbf s}_h^{j+1})
-(\bv^j-{\mbf s}_h^j)\right \|_0^2.
\end{align}
To bound the first term on the right-hand side of (\ref{tercero1}),
after taking Taylor expansion with integral reminder and applying
H\"older's inequality we
have
$$
\Delta t \sum_{j=0}^{n-1}t'_{j} \left\|\bv_t^{j+1}-\frac{\bv^{j+1}-\bv^j}{\Delta t}\right\|_0^2
\le \sum_{j=0}^{n-1} t'_j \int_{t_j}^{t_{j+1}}(s-t_j)^2\|\bv_{ss}\|_0^2.
$$
Now, for $j\ge 1$, we write $t_j'(s-t_j)^2=t_j(s-t_j)^2\le t_j {\Delta t}^2 \le s{\Delta t}^2$, and, for $j=0$, $t'_0(s-t_j)^2=\Delta t (s)^2\le s{\Delta t}^2$,
so that applying (\ref{eq:cota_vtt}) we get
\begin{equation}
\Delta t \sum_{j=0}^{n-1}t'_{j} \left\|\bv_t^{j+1}-\frac{\bv^{j+1}-\bv^j}{\Delta t}\right\|_0^2
\le {\Delta t}^2\int_{t_0}^{t_{n}}s\|\bv_{ss}\|_0^2
\le {\Delta t}^2 T \nu^4K_{4,3}^2.\label{cota_pre_l2_2}
\end{equation}
To bound the second term on the right-hand side of (\ref{tercero1}) we observe that
\begin{equation*}
\begin{split}
\left\|(\bv^{j+1}-{\mbf s}_h^{j+1})
-(\bv^j-{\mbf s}_h^j)\right \|_0^2&=\Bigl\|\int_{t_j}^{t_{j+1}}(\bv-{\mbf s}_h)_s~ds\Bigr\|_0^2\\
&\le \Delta t \int_{t_j}^{t_{j+1}}\left\|(\bv-{\mbf s}_h)_s\right\|_0^2~ds,
\end{split}
\end{equation*}
where, in the last inequality we have applied H\"older's inequality.
Now, for $j\ge 1$ we write $t'_j=t_j\le s$ and
apply~(\ref{eq:error_steady_simplified_linear}) to bound~$\left\|(\bv-{\mbf s}_h)_s\right\|_0^2$, and, for $j=0$, $t'_0=\Delta t$ and apply~(\ref{complemento24}), so that we have
\begin{align}
\frac{1}{\Delta t}\sum_{j=0}^{n-1} t_j'&\bigl\|(\bv^{j+1}-{\mbf s}_h^{j+1})
-(\bv^j-{\mbf s}_h^j)\bigr \|_0^2\nonumber\\
&\le C \int_{t_1}^{t_{n}}
{(h^4+\delta^2\nu^2)}s\left(\|\bv_s\|_2^2+\nu^{-2}\|q_s\|_1^2\right)~ds
\nonumber\\&\quad+C\int_0^{t_1}\left(\nu\Delta t^2+\frac{h^4}{\nu}+\nu\delta^2\right) \left(\|\bv_s\|_1^2+\nu^{-2}\|q_s\|_0^2\right)~ds\nonumber\\
&\le C(h^4+\nu^2(\Delta t^2+\delta^2)(\nu^2 K_{4,2}^2T+\nu K_{3,2}^2),\label{cota_pre_l2_3}
\end{align}
where in the last inequality we have applied~(\ref{eq:u-int}).
Thus, from~(\ref{tercero1}), (\ref{cota_pre_l2_2}) and (\ref{cota_pre_l2_3}) we finally reach
\begin{equation}\label{cota_pre_l2_4}
\Delta t \sum_{j=0}^{n-1}t_{j}'\|\tau_h^j\|_{0}^2\le C \left({\Delta t}^2\nu^4TK_{4,3}^2+(h^4+\nu^2(\Delta t^2+\delta^2)(\nu^2 K_{4,2}^2T+\nu K_{3,2}^2)\right),
\end{equation}
so that for the second term on the right-hand side of~(\ref{eq:auxa1})
we write
\begin{equation}\label{cota_pre_l2_4bosco}
\begin{split}
\Delta t \sum_{j=0}^{n-1}t_{j+1}^2\|\tau_h^j\|_{0}^2\le& Ct_n {\Delta t}^2\nu^4TK_{4,3}^2\\
&+Ct_n(h^4+\nu^2(\Delta t^2+\delta^2)(\nu^2 K_{4,2}^2T+\nu K_{3,2}).
\end{split}
\end{equation}
Let us also observe that by writing $\Delta t^{-1} t'_j\ge 1$ and using~(\ref{tercero1}), and repeating the arguments to~prove~(\ref{cota_pre_l2_3}), but~using~(\ref{complemento24}) instead of~(\ref{eq:error_steady_simplified_linear}) for $j\ge 1$ we get
\begin{equation}\label{cota_pre_l2_45}
\Delta t \sum_{j=0}^{n-1}\|\tau_h^j\|_{0}^2\le C \left({\Delta t}\nu^4T(K_{4,2}^2+K_{4,3}^3)+(h^2+\nu\delta)\nu^2 K_{3,2}^2\right),
\end{equation}

For the last term on the right-hand side of~(\ref{eq:auxa1}),
applying H\"older's inequality and~(\ref{eq:cota1}), we may write
%
\begin{align}
\delta \|\nabla (z_h^{j+1}-z_h^j)\|_{0}^2&= \delta\bigl\|\int_{t_j}^{t_{j+1}}(\nabla z_h)_s\bigr\|_0^2\le {\Delta t}\int_{t_j}^{t_{j+1}}\delta\|(\nabla z_h)_s\|_0^2~ds\nonumber\\
&\le C {\Delta t} \int_0^{t_1} (\nu\|\bv_s\|_1^2+\nu^{-1}\|q_s\|_0^2)~ds. 
\label{eq:dos_ast}
\end{align}
Thus, 
\begin{align}
\label{eq:cons_a_222}
{\Delta t}\sum_{j=0}^{n-1}
\|\nabla (z_h^{j+1}-z_h^j)\|_{0}^2  &\le \frac{\Delta t^2}{\delta} \int_0^{t_n} (\nu\|\bv_s\|_1^2+\nu^{-1}\|q_s\|_0^2)~ds,
\nonumber\\
\end{align}
and, consequently, for the last term on the right-hand side of~(\ref{eq:auxa1}), using also~(\ref{eq:u-int}), we have
\begin{align}
\label{eq:cons_a_2}
{\Delta t} \delta\sum_{j=0}^{n-1}
t_{j+1}\|\nabla (z_h^{j+1}-z_h^j)\|_{0}^2
&\le C\nu^{3}{\Delta t}^2 K_{3,2}^2 t_{n+1}.
\end{align}
To conclude we need to bound the first term on the right-hand side of (\ref{eq:auxa1}). For this purpose
we denote
$$
{\mbf S}_h(t)=\int_0^t {\mbf s}_h(s)\,ds, \ \widetilde {\mbf V}_h^n=\Delta t\sum_{j=1}^n\tilde \bv_h^j,\ P_h^n=\Delta t \sum_{j=1}^n q_h^j, \  \bG(t)=\int_0^t \bg(s)\,ds
$$
and~integrate~(\ref{eq:pro_stokes_evo})  with respect to~time taking into account that $\hat{\mbf  g}={\mbf g}-\bv_t$. Thus,
\begin{align}\label{eq:appr_stokes_evol1t}
(\bv(t),{\bs\chi}_h)+\nu(\nabla {\mbf S}_h,\nabla {\bs\chi}_h)+(\nabla \int_0^tz_h\,ds,{\bs\chi}_h)&=(\bG+\bv(0),{\bs\chi}_h),\quad \forall {\bs\chi}_h\in V_h,\\
(\nabla \cdot {\mbf S}_h,\psi_h)+\delta(\nabla \int_0^t z_h\,ds,\nabla \psi_h)&=0,\quad \forall \psi_h\in Q_h.\label{eq:appr_stokes_evol2t}
\end{align}
We also define
$$
\tilde \bE_h^n=\widetilde {\mbf V}_h^n-{\mbf S}_h(t_n), \qquad R_h^n= P_h^n-\int_0^{t_{n}}z_h\,ds,
$$
and
\begin{equation}\label{gammahj}
\begin{split}
\Upsilon_h^n=
(\tilde \bv_h^0-\bv(0))&+\Bigl(\Delta t\sum_{j=1}^{n+1} \bg(t_j)-\bG(t_{n+1})\Bigr)
\\&+\Bigl(\frac{1}{\Delta t}\int_{t_n}^{t_{n+1}}{\mbf s}_h(s)~ds-\bv(t_{n+1})\Bigr).
\end{split}
\end{equation}
We multiply~(\ref{eq:eu_non_tilde})-(\ref{eq:eu_non_tilde2}) by~$\Delta t$, sum from $j=0$ to $n$,
and subtract from~(\ref{eq:appr_stokes_evol1t})-(\ref{eq:appr_stokes_evol2t}) evaluated at $t=t_{n+1}$ to get
\begin{align*}
\left(\frac{\tilde \bE_h^{n+1}-\tilde \bE_h^n}{\Delta t},{\bs\chi}_h\right)+&\nu (\nabla \tilde \bE_h^{n+1},\nabla {\bs\chi}_h)+(\nabla R_h^n,{\bs\chi}_h)\nonumber\\
&{}=(\Upsilon_h^n,{\bs\chi}_h)-(\nabla D_h^n,{\bs\chi}_h),\quad\forall {\bs\chi}_h\in V_h,
\\
(\nabla \cdot \tilde \bE_h^{n+1},\psi_h)+\delta(\nabla& R_h^{n+1},\nabla \psi_h)=0,
\quad \psi_h\in Q_h,
\end{align*}
where
\begin{equation}
\label{eq:D_h}
D_h^n= {\Delta t} q_h^0 -\int_{t_n}^{t_{n+1}} z_h(s)\,ds=
{\Delta t}  r_h^0+{\Delta t}z_h(0)-\int_{t_n}^{t_{n+1}} z_h(s)\,ds.
\end{equation}
We notice that
$$
\frac{\tilde \bE_h^{n+1}-\tilde \bE_h^n}{\Delta t}=\tilde\be_h^{n+1}+\left({\mbf s}_h^{n+1}-\frac{1}{\Delta t}\int_{t_n}^{t_{n+1}}{\mbf s}_h(s)~ds\right),
\qquad \tilde \bE_h^0=0, \qquad R_h^0=0
$$so that
applying~(\ref{eq:stab_evol_2}) for $\bw_h^n=\tilde \bE_h^n$, $y_h^n=R_h^n$, $\bb_h^n=
\Upsilon_h^n$ and~$d_h^n=D_h^n$ we get
\begin{equation*}
\sum_{j=0}^{n-1}\Delta t \Bigl\|\frac{\tilde \bE_h^{n+1}-\tilde \bE_h^n}{\Delta t}\Bigr\|_0^2\le
 c_0\Delta t\sum_{j=0}^{n-1}(\|\Upsilon_h^j\|_0^2+\|\nabla D_h^j\|_0^2)\quad
\end{equation*}
and then
\begin{equation}
\label{eq:stab_evol_2I}
\sum_{j=1}^{n}\Delta t \|\tilde \be_h^n\|_0^2
\le
 C\Bigl(\Delta t\sum_{j=0}^{n-1}\bigl(\|\Upsilon_h^j\|_0^2+\|\nabla D_h^j\|_0^2+\bigl\|{\mbf s}_h^{j+1}
 -\frac{1}{\Delta t}\int_{t_j}^{t_{j+1}}{\mbf s}_h(s)~ds\bigr\|_0^2\bigr)\Bigr).
\end{equation}
We now bound the right-hand side of~(\ref{eq:stab_evol_2I}). We start with the second term of $\Upsilon_h^j$ in (\ref{gammahj}).
We notice that
$$
\int_{t_{j}}^{t_{j+1}} \bg(t)\,dt-\Delta t\bg(t_j)=\int_{t_{j}}^
{t_{j+1}} (\bg(t)-\bg(t_j))\,dt
$$
so that
by successively applying H\"older's inequality and the mean value theorem
we have
\begin{align*}
\bigl\|\int_{t_{j}}^{t_{j+1}} \bg(t)\,dt-\Delta t\bg(t_j)\bigr\|_0^2&\le \Delta t\int_{t_{j}}^{t_{j+1}}\|\bg(t)-\bg(t_j)\|_0^2\,dt\\
&\le
\Delta t \int_{t_{j}}^{t_{j+1}}\bigl\|\int_{t_j}^t\bg_s(s)\bigr\|_0^2\,dt
\\
&\le
\Delta t\int_{t_{j}}^{t_{j+1}}(t-t_j)\int_{t_j}^t \left\|\bg_s(s)\right\|_0^2\,ds\,dt
\\
&\le
 \Delta t \int_{t_j}^{t_{j+1}}\left\|\bg_s(s)\right\|_0^2\,ds\int_{t_{j}}^{t_{j+1}}(t-t_j)\,dt\\
 &\le
\frac{1}{2}{\Delta t}^{3} \int_{t_{j}}^{t_{j+1}}\left\|\bg_s(s)\right\|_0^2\,ds.
\end{align*}
Thus, applying H\"older's inequality we have
\begin{align*}
\Bigl\|\Delta t\sum_{l=1}^{j+1} \bg(t_l)-\bG(t_{j+1})\Bigr\|_0^2
&=
\Bigl\|\sum_{l=1}^{j+1}\Delta t \bg(t_l)-\int_{t_{l-1}}^{t_l}\bg(t)\,dt\Bigr\|_0^2
\\
&\le (j+1)\sum_{l=1}^{j+1}
\Bigl\|\Delta t \bg(t_l)-\int_{t_{l-1}}^{t_l}\bg(t)\,dt\Bigr\|_0^2
\\
&\le t_{j+1}\frac{1}{2}{\Delta t}^2
 \int_{0}^{t_{j+1}}\left\|\bg_t(t)\right\|_0^2\,dt.
\end{align*}
And then
\begin{equation}\label{lasgs}
\sum_{j=0}^{n-1}\Delta t\Bigl\|\sum_{l=1}^{j+1}\Delta t\bg(t_l)-\bG(t_{j+1})\Bigr\|_0^2\le Ct_n^2{{\Delta t}^2}\int_0^{t_n}\left\|\bg_t(t)\right\|_0^2\,dt\le Ct_n^2{\Delta t}^2\nu^2\hat K_3^2,
\end{equation}
where in the last inequality we have applied (\ref{eq:cota_gt}).

To bound the third term of $\Upsilon_h^j$ in (\ref{gammahj}) we observe that
\begin{equation}\label{eq:casitodo2}
\begin{split}
\sum_{j=0}^{n-1}\Delta t\bigl\|&\frac{1}{\Delta t}\int_{t_n}^{t_{n+1}}{\mbf s}_h(s)~ds-\bv(t_{n+1})\bigr\|_0^2
\\&=\sum_{j=0}^{n-1}\frac{1}{\Delta t}\bigl\|\int_{t_n}^{t_{n+1}}({\mbf s}_h(s)-{\Delta t}\bv(t_{n+1}))~ds\bigr\|_0^2\\
&\le\sum_{j=0}^{n-1}\frac{2}{\Delta t}\Bigl\|\int_{t_n}^{t_{n+1}}({\mbf s}_h(s)-{\Delta t}\mbf s(t_{n+1}))~ds\bigr\|_0^2\\
&\quad+\sum_{j=0}^{n-1}2{\Delta t}\bigl\|{\mbf s}_h(t_{n+1})-\bv(t_{n+1})\bigr\|_0^2.
\end{split}
\end{equation}
For the first term on the right-hand side of (\ref{eq:casitodo2}) arguing as in (\ref{lasgs}) and then applying (\ref{eq:cota0deri1}),
(\ref{eq:u-inf}) and (\ref{eq:u-int}) we finally get
\begin{align}\label{eq:casitodo3}
\sum_{j=0}^{n-1}\frac{2}{\Delta t}\Bigl\|\int_{t_n}^{t_{n+1}}({\mbf s}_h(s)-{\Delta t}\mbf s(t_{n+1}))~ds\Bigr\|_0^2&\le 2\sum_{j=0}^{n-1}{\Delta t}^2\int_{t_n}^{t_{n+1}}\|({\mbf s}_h)_s\|_0^2~ds\nonumber\\
&\le C {\Delta t}^2\nu^2\left( K_{3,2}^2h^2+M_{2,2}^2t_n\right).
\end{align}
For the second term on the right-hand side of (\ref{eq:casitodo2}) applying (\ref{eq:error_steady_simplified_linear}) and (\ref{eq:u-infa}) we get
\begin{eqnarray}\label{eq:casitodo4}
\sum_{j=0}^{n-1}2{\Delta t}\left\|{\mbf s}_h(t_{n+1})-\bv(t_{n+1})\right\|_0^2
&\le& 2t_n\max_{t_1\le s\le {t_n}}\left\|{\mbf s}_h(s)-\bv(s)\right\|_0^2
\nonumber\\
&
\le & Ct_n  (h^4+\delta^2\nu^2) M_{2,1}^2.
\end{eqnarray}
Inserting (\ref{eq:casitodo3}) and (\ref{eq:casitodo4}) in (\ref{eq:casitodo2}) we obtain
\begin{eqnarray}\label{second_in_gamma}
\sum_{j=0}^{n-1}\Delta t\Bigl\|\frac{1}{\Delta t}\int_{t_n}^{t_{n+1}}{\mbf s}_h(s)~ds-\bv(t_{n+1})\Bigr\|_0^2&\le& C {\Delta t}^2\nu^2\left( K_{3,2}^2h^2+M_{2,2}^2t_n\right)
\nonumber\\
&&\quad + Ct_n  (h^4+\delta^2\nu^2) M_{2,1}^2.
\end{eqnarray}
Then, from (\ref{lasgs}) and (\ref{second_in_gamma}) and taking into account the definition of $\Upsilon_h^j$ in (\ref{gammahj}) we finally reach
\begin{eqnarray}\label{eq:unadel_fin}
\Delta t\sum_{j=0}^{n-1}\|\Upsilon_h^j\|_0^2&\le& C\left( t_n\|\tilde \bv_h^0-\bv(0)\|_0^2+t_n^2{\Delta t}^{2} \nu^2 \hat K_3^2\right)\\
&&\quad+C \left({\Delta t}^2\nu^2\left( K_{3,2}^2h^2+M_{2,2}^2t_n\right)
+ t_n(h^4+\delta^2\nu^2) M_{2,1}^2\right)\nonumber.
\end{eqnarray}
To bound the second term on the right-hand side of (\ref{eq:stab_evol_2I}) we notice that using definition (\ref{eq:D_h}) we get
\begin{align*}
\|\nabla D_h^j\|_0^2&\le 4\Bigl\|\int_{t_j}^{t_{j+1}} \nabla z_h(s)\,ds\Bigr\|_0^2+4{\Delta t}^2\left\|\nabla z_h(0)
\right\|_0^2 +2{\Delta t}^2\left\|\nabla  r_h^0\right\|_0^2
\nonumber\\
&\le 4\Delta t \int_{t_j}^{t_{j+1}} \left \|\nabla z_h(s)\right\|_0^2\,ds +
4{\Delta t}^2\left\|\nabla z_h(0)
\right\|_0^2+2{\Delta t}^2\left\|\nabla   r_h^0\right\|_0^2.\nonumber\\
\end{align*}
Thus, applying (\ref{eq:cota_nabla_z_h}) and (\ref{eq:u-infa}) we obtain
\begin{eqnarray}\label{eq:dosdel_fin}
 \Delta t\sum_{j=0}^{n-1}\|\nabla D_h^j\|_0^2&\le& C{\Delta t}^2  \left(
\int_{0}^{t_{n}} \left \|\nabla z_h(s)\right\|_0^2\,ds
+t_n\left\|\nabla z_h(0)\right\|_0^2+t_n \left\|\nabla  r_h^0\right\|_0^2\right)\nonumber\\
&\le&C{\Delta t}^2t_n\left(\max_{0\le t\le t_n}\|\nabla z_h(t)\|_0^2+\left\|\nabla  r_h^0\right\|_0^2\right)
\nonumber\\
&\le&C{\Delta t}^2t_n\left(\nu^2 M_{2,1}^2+\left\|\nabla  r_h^0\right\|_0^2\right).
\end{eqnarray}
On the other hand, for the last term in~(\ref{eq:stab_evol_2I}) we get
\begin{align*}
\bigl\|{\mbf s}_h^{j+1}-\frac{1}{\Delta t}\int_{t_j}^{t_{j+1}}&{\mbf s}_h(s)~ds\bigr\|_0^2
=\bigl\|({\mbf S}_h)_t^{j+1}-\frac{{\mbf S}_h^{j+1}-{\mbf S}_h^j}{\Delta t}\bigr\|_0^2\\
&=\bigl\|\frac{1}{\Delta t}\int_{t_j}^{t_{j+1}} (t_{j}-s) ({\mbf S}_h)_{ss}\,ds\bigr\|_0^2\\
&\le\Delta t \int_{t_j}^{t_{j+1}} \left\|({\mbf s}_h)_{s}\right\|_0^2ds\\
&\le C{\Delta t}\int_{t_j}^{t_{j+1}} h^2\bigl(\|\nabla \bv_t\|_1^2+\nu^{-2}\|q_t\|_0^2~ds+\|\bv_t\|_0^2\,\bigr)ds,
\end{align*}
where in the last inequality we have applied (\ref{eq:cota0deri1}).
Consequently, using (\ref{eq:u-inf}) and (\ref{eq:u-int}) we reach
\begin{align}\label{eq:casitodo1}
\Delta t\sum_{j=0}^{n-1}\Bigl\|{\mbf s}_h^{j+1}&-\frac{1}{\Delta t}\int_{t_j}^{t_{j+1}}{\mbf s}_h(s)~ds\Bigr\|_0^2\nonumber\\
&\le C{\Delta t}^2
\int_{0}^{t_{n}} \left(h^2\left(\|\nabla \bv_t\|_1^2+\nu^{-2}\|q_t\|_0^2~ds\right)+\|\bv_t\|_0^2\,\right)ds\nonumber\\
&\le C {\Delta t}^2\nu^2 (K_{3,2}^2h^2 +M_{2,2}^2t_n).
\end{align}
Thus, in view of~(\ref{eq:stab_evol_2I}), (\ref{eq:unadel_fin}), (\ref{eq:dosdel_fin}) and~(\ref{eq:casitodo1}) we get
\begin{equation}
\label{eq:stab_evol_3I}
\begin{split}
\sum_{j=1}^{n}\Delta t\|\tilde\be_h^j\|_0^2\le& C\left( t_n\|\tilde \bv_h^0-\bv(0)\|_0^2+t_n{\Delta t}^2\left\|\nabla  r_h^0\right\|_0^2+\nu^2 t_n^2
\hat K_3^2{\Delta t}^{2}
\right)\\
+&C \left({\Delta t}^2\nu^2\left( K_{3,2}^2h^2+t_n(M_{2,1}^2+M_{2,2}^2)\right)
+ t_n  (h^4+\delta^2\nu^2) M_{2,1}^2\right).
\end{split}
\end{equation}
Going back to (\ref{eq:auxa1}) and inserting (\ref{cota_pre_l2_4bosco}), (\ref{eq:cons_a_2}) and  (\ref{eq:stab_evol_3I}) we finally reach
\begin{align}
\label{eq:auxa2bis}
t_n\|\tilde \be_h^{n}\|_0^2&
+\Delta t\sum_{j=1}^{n}t_{j}\bigl(\nu\|\nabla\tilde \be_h^{j}\|_0^2+{\delta}\|\nabla r_h^{j}\|_0^2\bigr)\nonumber
\\
&{}\le C\left( t_n\|\tilde \bv_h^0-\bv(0)\|_0^2+{\Delta t}\|\tilde \be_h^0\|_0^2+t_n({\Delta t})^2\left\|\nabla  r_h^0\right\|_0^2+t_n^2{\Delta t}^{2} \nu^2\hat K_3^2\right)
\nonumber\\
&\quad+ C\left(t_n (h^4+\delta^2\nu^2) (\nu^2K_{4,2}^2T+\nu K_{3,2}^2+  M_{2,1}^2)\right)\nonumber\\
&\quad+C{\Delta t}^2\left(\nu^4 T  (K_{4,2}^2+K_{4,3}^2)t_n+t_{n+1}\nu^{3} K_{3,2}^2\right)\nonumber\\
&\quad+C {\Delta t}^2\left(\nu^2\left( K_{3,2}^2h^2+t_n(M_{2,1}^2+M_{2,2}^2)\right)\right)\nonumber\\
&\le C t_n\left( \|\tilde \bv_h^0-\bv(0)\|_0^2+\|\tilde \be_h^0\|_0^2+{\Delta t}^2\left\|\nabla  r_h^0\right\|_0^2\right)\nonumber\\
&\quad+C_1 t_n {\Delta t}^2+C_2 t_n (h^4+\delta^2\nu^2),
\end{align}
where $C_1$ and $C_2$ are the constants in (\ref{laC1}) and (\ref{laC2}) and we have used the bounds  $t_{n+1}\le C t_n$, $\Delta t\le t_n$ and that $(\Delta t)^2\nu^2K_{3,2}^2h^2\le t_n(\Delta t)\nu^2 K_{3,2}^2h^2\le t_n((\Delta t)^2\nu^3+h^4\nu)K_{3,2}^2$.

To conclude (\ref{cota_th1_l2})  we apply (\ref{eq:auxa2bis}) together with triangle inequality, (\ref{eq:error_steady_simplified_linear})
and (\ref{eq:u-infa}).

Finally to prove~(\ref{cota_th1_H1}) we apply~(\ref{eq:stab_evol_1}) instead of~(\ref{eq:stab_evol_1t}). Then,
using~(\ref{eq:cota_menos1}) and then
applying~(\ref{cota_pre_l2_45}), (\ref{eq:cons_a_222}) and~(\ref{eq:u-int}),
we have that
\begin{align}\label{eq:segun_esti}
\Delta t&\sum_{j=1}^n\bigl(\nu \|\nabla \tilde\be_h^h\|_0^2+\delta\|\nabla r_h\|_0^2\bigr) \le
 C \left(\|\tilde\bv_h^0-\bv(0)\|_0^2 +M_{2,1}^2(h^4+(\nu\delta)^2)\right)
\nonumber\\
&\quad +\frac{C}{\nu\lambda}\Bigl(\nu^4(K_{4,2}^2+K_{4,3}^2)T\Delta t + \nu^2K_{3,2}(h^2+\nu\delta +\nu^2\lambda \Delta t^2\Bigr)
\nonumber\\
&\le C(\|\tilde\bv_h^0-\bv(0)\|_0^2 \nonumber\\
&\quad +C\bigl(C_1\Delta t\bigl((\nu\lambda)^{-1} +\Delta t\bigr)
+ C_2(\lambda^{-1}(h^2+\nu\delta) + h^4+(\nu\delta)^2\bigr )\nonumber\\
&\le C \|\tilde\bv_h^0-\bv(0)\|_0^2 +\tilde C_1 \Delta t+\tilde C_2(h^2+\nu\delta),
\end{align}
where $\tilde C_1$ and $\tilde C_2$ are the constants in (\ref{laCtilde1}) and (\ref{laCtilde2}). Taking into account~(\ref{eq:error_steady_simplified_linear}) the estimate
(\ref{cota_th1_H1}) follows.

\end{proof}
\begin{remark}\label{remark1} Let us observe that taking $\delta=\Delta t$ the analysis carried out applies to  the standard Euler non-incremental scheme. However, since condition (\ref{eq:cond_delta}) implies
\begin{equation}\label{eq:cond_pidomask}
\frac{1}{\nu\rho_1^2} h^2\le \Delta t,
\end{equation}
the analysis for the standard Euler non-incremental scheme holds under condition (\ref{eq:cond_pidomask}). This result is in agreement with the
error bounds in \cite{badia-codina} where the authors prove error bounds for the Euler non-incremental scheme for LBB stable elements assuming $\Delta t\ge C h^2$. It is also in agreement with the classical results for the continuous in space Euler non-incremental method (see for example \cite{Guermond_overview}) since for $\delta=\Delta t$ the rate of convergence in terms of $\Delta t$ in the $L^2$ norm of the velocity is one and the rate of convergence in the $H^1$ norm of the velocity and the $L^2$ norm of the pressure is one half, see (\ref{cota_th1_l2})-(\ref{cota_th1_H1}).
\end{remark}
\begin{remark}\label{remark2}
In view of (\ref{cota_th1_l2}) and (\ref{cota_th1_H1}) any initial approximation $\tilde \bv_h^0$  based on linear elements such us the linear
interpolant of the initial condition gives the optimal order for the term $\|\tilde \bv_h^0-\bv(0)\|_0$. For the initial pressure
any initial pressure giving $\|\nabla r_h^0\|_0=O(1)$ keeps the optimal rate of convergence. In particular, choosing $q_h^0=0$ which means taking in (\ref{eq:eu_non}) $\bv_h^0=\tilde\bv_h^0$
gives $\|\nabla r_h^0\|_0=\|\nabla z_h\|_0$ that is bounded thanks to (\ref{eq:cota_nabla_z_h}).
\end{remark}
\begin{remark} As commented in~\cite{Javier-Bosco-JuliaI},
the restriction~(\ref{eq:cond_delta2}) for the modified Euler non-incremental scheme is not just a requirement of the proof but, as it can be easily checked in practice, the method becomes unstable if $\Delta t$
is taken larger than $2\delta$.
\end{remark}
We will now prove a bound for the pressure error.
\begin{theorem}\label{Th2} Under the assumptions of Theorem~\ref{Th1} the following bound holds
\begin{align}\label{cota_prefinal_mod}
\Delta t\sum_{j=1}^nt_j\|q_h^j-q(t_j)\|_0^2\le  &C(t_{n+1}\nu+\lambda^{-1}) \|\tilde\bv_h^0-\bv(0)\|_0^2\nonumber\\
&{}+C t_{n+1}\nu\left(\|\tilde \be_h^0\|_0^2+{\Delta t}^2\|\nabla r_h^0\|_0^2\right)\nonumber\\
&{}+ C_3\Delta t + C_4 (h^2+\nu\delta),
\end{align}
where, (assuming (\ref{delta<T})),
\begin{align}
\label{laC3}
C_3&=C(\nu T+\lambda^{-1}) \tilde C_1\\
C_4&=C\bigl(\nu T+\lambda^{-1})\tilde C_2,
\label{laC4}
\end{align}
and where~$\tilde C_1$ and~$\tilde C_2$ are the constants in~(\ref{laCtilde1}) and~(\ref{laCtilde2}) respectively.
\end{theorem}
\begin{proof}
Applying Lemma~\ref{lema_presion} and (\ref{eq:er3_orig}) it is easy to obtain
\begin{equation}\label{eq:conlema1_1}
\begin{split}
\Delta t \sum_{j=1}^{n}t_{j}\|r_h^{j}\|_0^2&\le C\Delta t \sum_{j=1}^{n}t_{j}\nu\delta \|\nabla r_h^{j}\|_0^2\\
&\quad+C\Delta t \sum_{j=1}^{n}t_{j}\Bigl\|\frac{\tilde \be_h^{j+1}-\tilde \be_h^j}{\Delta t}\Bigr\|_{-1}^2
+C\Delta t \sum_{j=1}^{n}t_{j}\|\tau_h^j\|_{-1}^2\\
&\quad+C\Delta t \sum_{j=1}^{n}t_{j}\nu^2\|\nabla \tilde e_h^{j+1}\|_0^2+C\Delta t \sum_{j=1}^{n}t_{j}\|(z_h^j-z_h^{j+1})\|_0^2.
\end{split}
\end{equation}
We will bound all the terms on the right-hand side of (\ref{eq:conlema1_1}). We first observe that the first and forth terms
are already bounded in (\ref{eq:auxa2bis}) and then
 \begin{align}\label{cota_pre_l2_1}
 C\Delta t &\sum_{j=1}^{n}t_{j}\nu\delta \|\nabla r_h^{j}\|_0^2+C\Delta t \sum_{j=1}^{n}t_{j}\nu^2\|\nabla \tilde e_h^{j+1}\|_0^2\nonumber\\
 &\le C t_{n+1}\nu \left(\|\tilde\bv_h^0-\bv(0)\|_0^2+\|\tilde \be_h^0\|_0^2+{\Delta t}^2\|\nabla r_h^0\|_0^2\right)\nonumber\\
 &\quad+\nu t_{n+1}\bigl(C_1{\Delta t}^2+C_2(h^4+(\nu\delta)^2)\big).
 \end{align}
 To bound the third term on the right-hand side of (\ref{eq:conlema1_1}) we first apply~(\ref{eq:cota_menos1}) and then~(\ref{cota_pre_l2_4}) to get
\begin{equation}\label{tercero}
\begin{split}
\Delta t \sum_{j=1}^{n}&t_{j}\|\tau_h^j\|_{-1}^2\le \frac{C}{\lambda}\Delta t \sum_{j=1}^{n}t_{j}\|\tau_h^j\|_{0}^2\\
&\le \frac{C}{\lambda} \left({\Delta t}^2\nu^4TK_{4,3}^2+(h^4+\nu^2(\Delta t^2+\delta^2))(\nu^2 K_{4,2}^2T+\nu K_{3,2}^2)\right)\\
&\le {C}\lambda^{-1} \left(C_1\Delta t^2+C_2(h^4+(\nu\delta)^2)\right).
\end{split}
\end{equation}
For the last term on the right-hand side of (\ref{eq:conlema1_1}) arguing as in~(\ref{eq:dos_ast}) and applying (\ref{eq:cota_z_h_0}) to the time derivative $\|(z_h)_t\|_0$, and (\ref{eq:u-int}) we get
\begin{equation}\label{la116}
\begin{split}
\Delta t \sum_{j=1}^{n}t_{j}\|(z_h^j-z_h^{j+1})\|_0^2&\le  {\Delta t}^2 \sum_{j=1}^nt_j\int_{t_j}^{t_{j+1}}\|(z_h)_s\|_0^2~ds\\
&\le {\Delta t}^2Ct_n\int_{t_1}^{t_{n+1}}s\left(\nu^2\|\bv_s\|_1^2+\|q_s\|_0^2\right)
\\&
\le Ct_{n}{\Delta t}^2 \nu^4 K_{3,2}^2\\
&\le C\nu T C_1\Delta t^2.
\end{split}
\end{equation}
To conclude we will bound the second term on the right-hand side of (\ref{eq:conlema1_1}).
Since $t_j\le t_{j+1}$ and taking into account (\ref{eq:cota_menos1}) we can write
\begin{eqnarray*}
\Delta t \sum_{j=1}^{n}t_{j}\Bigl\|\frac{\tilde \be_h^{j+1}-\tilde \be_h^j}{\Delta t}\Bigr\|_{-1}^2
&\le&  \Delta t \sum_{j=1}^{n}t_{j+1}\Bigl\|\frac{\tilde \be_h^{j+1}-\tilde \be_h^j}{\Delta t}\Bigr\|_{-1}^2\nonumber\\
&\le&
C \lambda^{-1}\Delta t \sum_{j=1}^{n}t_{j+1}\Bigl\|\frac{\tilde \be_h^{j+1}-\tilde \be_h^j}{\Delta t}\Bigr\|_{0}^2.
\end{eqnarray*}
Applying (\ref{unamas}) we get
\begin{align}
\lambda^{-1}\Delta t \sum_{j=1}^{n}t_{j+1}&\Bigl\|\frac{\tilde \be_h^{j+1}-\tilde \be_h^j}{\Delta t}\Bigr\|_{0}^2\nonumber\\
&\le c_0\lambda^{-1}
\Bigl( \Delta t\sum_{j=1}^nt_{j+1}\|\tau_h^j\|_0^2+ \Delta t\sum_{j=1}^nt_{j+1}\|\nabla (z_h^{j+1}-z_h^j)\|_0^2\nonumber\\
&\quad+\Delta t\nu\sum_{j=1}^n\|\nabla \tilde \be_h^j\|_0^2
+\Delta t \delta \sum_{j=1}^n\|\nabla r_h^j\|_0^2\Bigr).\label{cota_pre_l2_5}
\end{align}
To conclude we will bound the four terms on the right-hand side of (\ref{cota_pre_l2_5}).
For the first one recalling that $t_{j+1}/t_j\le 2$ for $j\ge 1$ and applying  (\ref{cota_pre_l2_4}) we get
\begin{align}\label{cota_pre_l2_6}
\Delta t \sum_{j=1}^{n}t_{j+1}\|\tau_h^j\|_{0}^2 & \le C \left({\Delta t}^2\nu^4TK_{4,3}^2+(h^4+\nu^2
(
\Delta t^2
+\delta^2))(\nu^2 K_{4,2}^2T+\nu K_{3,2}^2)\right)
\nonumber\\
&\le {C}\lambda^{-1} \left(C_1\Delta t^2+C_2(h^4+(\nu\delta)^2)\right).
\end{align}
To bound the second term on the right-hand side of (\ref{cota_pre_l2_5}) arguing as usual and applying (\ref{eq:cons_a_222}),(\ref{eq:cond_delta2})
and (\ref{eq:u-int}) we get
\begin{align}\label{cota_pre_l2_7}
\Delta t\sum_{j=1}^nt_{j+1}\|\nabla (z_h^{j+1}-z_h^j)\|_0^2
&\le Ct_{n+1}{\Delta t}\nu^3 K_{3,2}^2\le CTC_1\Delta t \le C\tilde C_1 \Delta t.
\end{align}
To conclude we observe that the last two terms in (\ref{cota_pre_l2_5}) have been bounded in (\ref{eq:segun_esti}). Then
\begin{align}\label{eq:forma0}
\Delta t\nu\sum_{j=1}^n\|\nabla \tilde \be_h^j\|_0^2+\Delta t \delta \sum_{j=1}^n\|\nabla r_h^j\|_0^2 \le   C \|\tilde\bv_h^0-\bv(0)\|_0^2 +\tilde C_1 \Delta t+\tilde C_2(h^2+\nu\delta).
\end{align}
Thus, inserting~(\ref{cota_pre_l2_6}), (\ref{cota_pre_l2_7}) and~(\ref{eq:forma0}) into~(\ref{cota_pre_l2_5}) we have
\begin{align}\label{cota_pre_l2_11_f2}
\lambda^{-1}\Delta t \sum_{j=1}^{n}t_{j+1}\Bigl\|\frac{\tilde \be_h^{j+1}-\tilde \be_h^j}{\Delta t}\Bigr\|_{0}^2
\le& C\lambda^{-1}\bigl(\|\tilde\bv_h^0-\bv(0)\|_0^2 +\tilde C_1\Delta t \bigr)\nonumber\\
&{}+C\lambda^{-1}\bigl((h^2+\nu\delta) \tilde C_2+C_1 \Delta t^2\bigr)\\
 &+\lambda^{-1}C_2(h^4+(\nu\delta)^2)
 \nonumber\\
{}\le& C\lambda^{-1}\bigl(\|\tilde\bv_h^0-\bv(0)\|_0^2 +\tilde C_1\Delta t+ \tilde C_2(h^2+\nu\delta)\bigr).\nonumber
\end{align}
Finally, inserting (\ref{cota_pre_l2_1}), (\ref{tercero}), (\ref{la116}) and (\ref{cota_pre_l2_11_f2}) in (\ref{eq:conlema1_1})  we obtain
for the modified Euler non-incremental method
\begin{align}\label{a_ver_si_esta}
\Delta t\sum_{j=1}^nt_j\|r_h^j\|_0^2 \le &C(t_{n+1}\nu+\lambda^{-1}) \|\tilde\bv_h^0-\bv(0)\|_0^2+C t_{n+1}\nu\left(\|\tilde \be_h^0\|_0^2+{\Delta t}^2\|\nabla r_h^0\|_0^2\right)\nonumber\\
&{}+C\lambda^{-1}\bigl( \tilde C_1\Delta t+\tilde C_2(h^2+\nu\delta)\bigr)\nonumber\\
 &{}+C(\nu t_{n+1}+\lambda^{-1})\bigl(C_1{\Delta t}^2+C_2(h^4+(\nu\delta)^2)\bigr)
\nonumber \\
 {}\le&\,\, C(\nu T+\lambda^{-1})\bigl( \|\tilde\bv_h^0-\bv(0)\|_0^2+\|\tilde \be_h^0\|_0^2+{\Delta t}^2\|\nabla r_h^0\|_0^2\bigr)
 \nonumber\\
&{}+C(\nu T+\lambda^{-1})\bigl(\tilde C_1\Delta t +\tilde C_2(h^2+\nu\delta) \bigr)
 \end{align}
Now observing that due to~(\ref{eq:error_steady_simplified_linear}) we have
$$
\|z_h^j - q(t_j) \|_0^2 \le \nu^2 M_{2,1}^2 (h^2+\nu\delta)\le \nu^2 C_2(h^2+\nu\delta) \le (\nu/T) \tilde C_2(h^2+\nu\delta),
$$
for $j=0,1,\ldots,N$, applying the triangle inequality in~(\ref{a_ver_si_esta}),  we finally reach~(\ref{cota_prefinal_mod}).
\end{proof}
\begin{remark}
Let us observe that any initial approximation for the velocity such that $\|\tilde \bv_h^0-\bv(0)\|_0=O(h^2)$ and
any initial approximation for the pressure satisfying $\|\nabla r_h^0\|_0=O(1)$ (which includes the choice $q_h^0=0$)
keep the optimal rate of convergence for the pressure $O(h^2+\Delta t)$ for  the  Euler non-incremental method
or $O(h^2+\delta +\Delta t)$ for the modified Euler non-incremental method.

We also observe that the error bounds (\ref{cota_th1_l2}) and~(\ref{cota_th1_H1})
for the Euler non-incremental method hold under assumption (\ref{eq:cond_pidomask}),
$h^2/(\nu\rho_1^2)\le \Delta t $. This means that for the Euler non-incremental method  $\Delta t =O(h)$ could be a possible choice.

On the other hand, the bounds (\ref{cota_th1_l2}) and~(\ref{cota_th1_H1}) and (\ref{cota_prefinal_mod})
 for the modified Euler non-incremental method hold only under assumption
$\Delta t\le \delta$. Then, for the modified Euler non-incremental one can choose $\Delta t$ as small as possible, and, in particular, one can make $\Delta t \rightarrow 0$.
\end{remark}
\section{Navier-Stokes equations}
We now consider the following initial value problem associated with the Navier-Stokes equations.
\begin{align}
\label{NS} \partial_t\bu -\nu \Delta \bu + (\bu\cdot\nabla)\bu + \nabla p &= \bff &&\text{in }\ (0,T)\times\Omega,\nonumber\\
\nabla \cdot \bu &=0&&\text{in }\ (0,T)\times\Omega,\\
\bu(0, \cdot) &= \bu_0(\cdot)&&\text{in }\ \Omega,\nonumber
\end{align}
and its discretization by the modified semi implicit Euler non-incremental method,
\begin{align}\label{eq:ns_non_tilde}
\left(\frac{\tilde \bu_h^{n+1}-\tilde \bu_h^n}{\Delta t},{\bs\chi}_h\right)+\nu(\nabla \tilde \bu_h^{n+1},\nabla {\bs\chi}_h)+
&(B(\tilde \bu_h^{n},\tilde \bu_h^{n}),{\bs\chi}_h)+(\nabla p_h^n,{\bs \chi}_h)
\nonumber\\
&{}=({\mbf g}^{n+1},{\bs\chi}_h),\qquad \forall {\bs\chi}_h\in V_h\qquad\\
(\nabla \cdot \tilde \bu_h^{n+1},\psi_h)+\delta(\nabla p_h^{n+1},\nabla \psi_h)&=0,\qquad \forall \psi_h\in Q_h,\nonumber
\end{align}
together with the initial condition to be specified later.
In~(\ref{eq:ns_non_tilde}) and in the sequel, $B(\cdot,\cdot)$ denotes the
following bilinear form
$$
B(\bv,\bw)=\bu\cdot\nabla\bw + \frac{1}{2}(\nabla\cdot \bv)\bw,
\qquad \bv,\bw\in V.
$$
Notice the well-known skew-symmetric property,
\begin{equation}
\label{eq:skew}
(B(\bv,\bw),{\bf y})=-(B(\bv,{\bf y}),\bw),\qquad \bv,\bw,{\bf y}\in V,
\end{equation}
so that in particular, $(B(\bv,\bw),\bw)=0$.

The numerical approximation $(\tilde \bu_h^n,p_h^n)$ of (\ref{eq:ns_non_tilde}) will be compared with the solution $(\tilde \bv_h^n,q_h^n)$
of~(\ref{eq:eu_non_tilde})-(\ref{eq:eu_non_tilde2}) for
\begin{equation}
\label{eq:ns_lag}
\bg = {\bf f} - B(\bu,\bu).
\end{equation}
On the other hand, along this section we apply to  $(\tilde \bv_h^n,q_h^n)$ the error bounds obtained in the previous section where  $(\tilde \bv_h^n,q_h^n)$
is compared with the stabilized
Stokes approximation $({\mbf s}_h^n,z_h^n)$ defined in (\ref{eq:pro_stokes})-(\ref{eq:pro_stokes2}) for
\begin{equation}\label{eq:lag}
\hat\bg=\bg-{\bu}_t.
\end{equation}
Whenever $\left\| \bg_t\right\|^2$ is integrable
in~$(0,T]$, i.e. the constant $\hat K_3^2$ in (\ref{eq:cota_gt}) is finite, this approximation will satisfy the error bounds (\ref{cota_th1_l2}), (\ref{cota_th1_H1}) and (\ref{cota_prefinal_mod}).

In the rest of this section we shall assume that ${\bf f},{\bf f}_t,{\bf f}_{tt}\in L^2(0,T]$ and that
$\bv=\bu$, $q=p$ satisfies the bounds~(\ref{eq:M1-M2}--\ref{eq:cota_vtt}). Since we now prove $\hat K_3^2$ in (\ref{eq:cota_gt}) is finite
 all appearances of the constants in~(\ref{eq:M1-M2}--\ref{eq:cota_gt}) will be for $\bv=\bu$ and $q=p$.

To prove $\hat K_3^2$ is finite we first observe that $\bg_t = {\bf f}_t - B(\bu,\bu)_t\in L^2(0,T)$ if
$\bu\cdot \nabla \bu_t \in L^2(0,T)$ and~$\bu_t\nabla \bu\in L^2(0,T)$.
We now show that this is so for the more difficult case~$d=3$. Applying~(\ref{eq:agmon}) we have
$$
\left\| \bu\cdot \nabla \bu_t \right\|_0 \le
\left\| \bu\right\|_\infty\left\| \nabla \bu_t \right\|_0
\le c_{\mathrm{A}} \left\| \bu\right\|^{1/2}_1 \left\|\bu\right\|_2^{1/2}\left\|\nabla \bu_t\right\|_0
\le c_{\mathrm{A}} M_1^{1/2}M_2^{1/2}\|\nabla \bu_t\|_0.
$$
Similarly, applying H\"older's inequality with~$p=d$ and
$q=d/(d-1)$, and then~(\ref{eq:sob}) and~(\ref{eq:parti_ineq}), we have
\begin{align*}
\left\| \bu_t\cdot \nabla \bu \right\|_0 \le
\left\| \bu_t\right\|_{L^6}\left\| \nabla \bu \right\|_{L^3}
&\le c_{1}^{3/2}\left\|\nabla\bu_t\right\|_0 \left\|\nabla \bu\right\|_0^{1/2}
\left\|\bu\right\|_2^{1/2}\nonumber\\
&\le c_{1}^{3/2}M_1^{1/2}M_2^{1/2}
\left\|\nabla\bu_t\right\|_0,
\end{align*}
so that $\left\| \bu\cdot\nabla \bu_t\right\|_{L^2(0,T)}+\left\| \bu_t\cdot\nabla \bu\right\|_{L^2(0,T)}\le C\nu (M_1M_2)^{1/2}K_{3,2}$
and consequently $\hat K_3^2$ in (\ref{eq:cota_gt}) is finite.
 We can state the following result.
\begin{theorem}\label{Th3} Let $(\bu,p)$ be the solution of (\ref{NS}) and let $(\tilde \bv_h^n,q_h^n)$  be the solution
of (\ref{eq:eu_non_tilde})-(\ref{eq:eu_non_tilde2}) with $\bg$ defined in (\ref{eq:ns_lag}) and $(\tilde \bv_h^0,q_h^0)=({\mbf s}_h^0,z_h^0)$.
Under the assumptions of Theorem~\ref{Th1} the following bounds hold
\begin{eqnarray}\label{cota_th1_l2_bis}
&&\max_{0\le t_n\le T}\|\tilde \bv_h^n-\bu(t_n)\|_0^2\le C_1 (\Delta t)^2+C_2 (h^4+(\nu\delta)^2),
\end{eqnarray}
where $C_1$ and $C_2$ are the constants in (\ref{laC1}) and (\ref{laC2}), and, assuming for simplicity that~(\ref{delta<T}) holds,
\begin{align}\label{cota_th1_h1_bis}
&\Delta t\sum_{j=1}^{n}\bigl(\nu\|\nabla (\tilde \bv_h^j-\bu(t_j))\|_0^2+\delta \|q_h^j -p(t_j)\|_0^2\bigr)\le \tilde C_1\Delta t + \tilde C_2(h^2+\nu\delta),\\
&\Delta t\sum_{j=1}^{n}t_j\bigl(\nu\|\nabla (\tilde \bv_h^j-\bu(t_j))\|_0^2+\delta \|q_h^j -p(t_j)\|_0^2\bigr)\le  C_1t_n\Delta t^2
\nonumber\\
&\quad+\tilde C_2 t_n(h^2+\nu\delta).
\label{cota_th1_H1_bis}
\end{align}
where $\tilde C_1$ and $\tilde C_2$ are the constants in (\ref{laCtilde1}) and (\ref{laCtilde2}).
\end{theorem}
\begin{proof}
The error bounds  (\ref{cota_th1_l2_bis}),  (\ref{cota_th1_h1_bis}) follow from (\ref{cota_th1_l2}) and (\ref{cota_th1_H1}) respectively
taking into account that applying (\ref{eq:error_steady_simplified_linear}) and
(\ref{eq:u-infa}) $\|\tilde \bv_h^0-\bu(0)\|_0^2\le C M_{2,1}^2 (h^4+\delta^2\nu^2)$ and $\nu\|\nabla(\tilde \bv_h^0-\bu(0))\|_0^2\le C \nu M_{2,1}^2 (h^2+\delta\nu)$ and that $\|\nabla r_h^0\|_0=0$.
The error bound (\ref{cota_th1_H1_bis}) follows from~(\ref{eq:auxa2bis}), the decomposition
$\tilde \bv_h^j-\bu(t_j)=\tilde\be_h^j +{\mathbf s}_h(t_j)-\bu(t_j)$ and~$q_h^j-p(t_j))=r_h^j + z_h(t_j)-p(t_j)$ and the error
bound~(\ref{eq:error_steady_simplified_linear}).
\end{proof}

The error bounds for the discretization~(\ref{eq:ns_non_tilde}) will be
obtained as a consequence of several previous results that we now
state. The first one is a discrete Gronwall lemma whose proof can be
easily obtained by induction (see e.g., \cite{heyran4}).

\begin{lema}\label{le:gronwall}
Let $k$, $B$, and $a_n,b_n,c_n,\gamma_n$ be nonnegative numbers such that
$$
a_n+k\sum_{j=0}^n b_j\le k\sum_{j=0}^{n-1}\gamma_j a_j+k\sum_{j=0}^nc_j+B,\quad n\ge 0.
$$
Then, the following bound holds
$$
a_n+k\sum_{j=0}^n b_j\le \exp\biggl(k\sum_{j=0}^{n-1}\gamma_j\biggr)\biggl(k\sum_{j=0}^n c_j +B\biggr),\quad n\ge 0.
$$
\end{lema}

\begin{remark}\label{re:gronwall} The statement of Lemma~\ref{le:gronwall} above is very similar to~Lemma~5.1 in~\cite{heyran4}, where the sum involving the terms
$\gamma_ja_j$ includes also the term $\gamma_n a_n$. In order to extend the analysis in the present paper to~the fully
implicit backward Euler method, Lemma~\ref{le:gronwall} must be replaced by~\cite[Lemma~5.1]{heyran4}.
\end{remark}

\begin{lema}\label{le:nonlin} For~$\bv,\bw,{\bs\phi}\in V$  the following bounds hold
\begin{align}
\label{eq:nonlin1}
\|B(\bv,\bv)-B(\bw,\bw)\|_{0}
\le&\bigl(\|\nabla\bv\|_{L^{2d/(d-1)}} +
\|\nabla \bw\|_{L^{2d/(d-1)}}\bigr) \left\|\be\right\|_{L^{2d}}
\nonumber\\
&{}+
\bigl(\left\|\bv\right\|_\infty
+\left\|\bw\right\|_\infty\bigr)\left\|\nabla\be\right\|_0.
\end{align}
\begin{align}
\label{eq:nonlin2}
|(B(\bv,\bv)-B(\bw,\bw),{\bs\phi})|
\le\left\|\be\right\|_{0}\Bigl(&\bigl(\|\nabla\bv\|_{L^{2d/(d-1)}} +
\|\nabla \bw\|_{L^{2d/(d-1)}}\bigr) \left\|{\bs\phi}\right\|_{L^{2d}}
\nonumber\\
&{}+
\bigl(\left\|\bv\right\|_\infty
+\left\|\bw\right\|_\infty\bigr)\left\|\nabla{\bs\phi}\right\|_0\Bigr),
\end{align}
where ${\bs e}=\bv-\bw$.
\end{lema}
\begin{proof}
From the identity
\begin{equation}
\label{eq:unadifB}
B(\bv,\bv)-B(\bw,\bw)=B(\be,\bv)+ B(\bw,\be),
\end{equation}
and applying H\"older inequality we have
\begin{align*}
|B(\be,\bv)| &\le \left\|\be \right\|_{L^{2d}}\|\nabla\bv\|_{L^{2d/(d-1)}}
+\frac{1}{2}\left\|\nabla\cdot\be\right\|_0\left\|\bv\right\|_\infty,
\nonumber\\
|B(\bw,\be)| &\le \left\|\bw \right\|_\infty\|\nabla\be\|_0
+\frac{1}{2}\left\|\nabla\cdot\bw\right\|_{L^{2d/(d-1)}}\left\|\be\right\|_{L^{2d}},
\end{align*}
and the bound~(\ref{eq:nonlin1}) follows.
To prove~(\ref{eq:nonlin2}),
we multiply~(\ref{eq:unadifB}) by~${\bs\phi}\in H^1_0$ and integrate in~$\Omega$, integrating by parts adequately and using the skew-symmetry property~(\ref{eq:skew}) we have
\begin{align*}
(B(\bv,\bv)-B(\bw,\bw),{\bs\phi})=&
\frac{1}{2}\bigl((\be\cdot\nabla \bv,{\bs\phi}) -( \be\cdot\nabla{\bs\phi},
 \bv)\bigr)\nonumber\\
 &{} + (B(\bw,{\bs\phi}),\be).
\end{align*}
and the bound follows by applying H\"older inequality~(\ref{eq:tri_holder}).
\end{proof}
\begin{lema}\label{le:cotas_threshold} Let  $(\tilde \bv_h^n,q_h^n)$  be the solution
of (\ref{eq:eu_non_tilde})-(\ref{eq:eu_non_tilde2}) with $\bg$ defined in (\ref{eq:ns_lag})
and $(\tilde \bv_h^0,q_h^0)=({\mbf s}_h^0,z_h^0)$. Under the assumptions of Theorem~\ref{Th1},
 and assuming also
\begin{equation}\label{eq:cota_deltat_up}
\nu \delta \le c_M\hbox{\rm diam}(\Omega) h,
\end{equation}
for a scale-invariant $c_M>0$,
there exists a scale invariant constant $c_{\mathrm{r}}>0$ depending on~the constants~$c_{\mathrm{inv}}$, $c_P$ and~$c_A$  in~(\ref{inv}), (\ref{eq:poincare2}) and~(\ref{eq:agmon}), respectively,
and constant $C_{\mathrm{th}}>0$ depending also on the constants
in~(\ref{cota_th1_h1_bis}), $\nu^{-1}$, $T$, the constants  $M_1$ and $M_2$ in (\ref{eq:M1-M2}), the
constant~$c_{\mathrm{int}}$
in~(\ref{cota_inter}), and the constant $\Lambda$ in~(\ref{eq:quasi}) (and also on $\max_{0\le t\le T}\|\bu(t)\|_0$ in the case $d=2$)
such that the following bounds hold for
all sequences $(\bw_h^n)_{n=0}^{N=T/\Delta t}$ in~$V_h$,
\begin{align*}
\Delta t\sum_{j=0}^n \| \nabla \bw_h^j\|_{L^{2d/(d-1)}}^2
 &\le\frac{c_{\mathrm{r}}\Lambda}{h} \Delta t\sum_{j=0}^n
{\|\nabla( \bw_h^j - \tilde \bv_h^j)\|_0^2} +
C_{\mathrm{th}},
\nonumber\\
\Delta t\sum_{j=0}^n\|  \bw_h^j\|_{\infty}^2&\le
\biggl(\frac{c_{\mathrm{r}}\Lambda}{h} \Delta t\sum_{j=0}^n{\| \nabla( \bw_h^j - \tilde \bv_h^j)\|_0^2}+
C_{\mathrm{th}}\biggr)\left|\Omega\right|^{(3-d)/{(2d)}},
\end{align*}
and $n=0,1,\ldots,N=T/\Delta t$.
\end{lema}
\begin{proof} We start with the first bound. We
write
\begin{equation}
\label{ayu1}
\bw_{h}^j= (\bw_{h}^j - \tilde \bv_h^j) + (\tilde \bv_h^j - I_h(\bu(t_j))) +
(I_h(\bu(t_j))) - \bu(t_j)) + \bu(t_j)
\end{equation}
We notice that for~$p=2d/(d-1)$ and~$q=2$ we have
$$
d\biggl(\frac{1}{q}-\frac{1}{p}\biggr)=
d\biggl(\frac{1}{2}-\frac{d-1}{2d}\biggr)=\frac{1}{2},
$$
so that applying~(\ref{inv}) with~$m=1$, $p=2d/(d-1)$, $q=2$ and~$l=1$, and using~(\ref{eq:quasi}) and (\ref{ayu1})
we get
\begin{align*}
\left\| \nabla \bw_h^j\right\|_{L^{2d/(d-1)}}
\le& c_{\mathrm{inv}}
\frac{\| \bw_h^j - \tilde \bv_h^j\|_1+\| \tilde \bv_h^j-I_h(\bu(t_j))\|_1}{(h/\Lambda)^{1/2}}\nonumber\\
 &{}+
\left\| \nabla( I_h(\bu(t_j)) - \bu(t_j))\right\|_{L^{2d/(d-1)}}
+
\left\| \nabla\bu(t_j)\right\|_{L^{2d/(d-1)}}.
\end{align*}
We notice that due to the interpolation bound~(\ref{cota_inter})
we have
$$\left\| \nabla( I_h(\bu(t_j)) - \bu(t_j))\right\|_{L^{2d/(d-1)}}\le
c_{\mathrm{int}}h^{1/2}\left\|\bu(t_j)\right\|_{2}\le c_{\mathrm{int}}h^{1/2}M_2,
$$
and due to~(\ref{eq:parti_ineq}),
$\left\| \nabla\bu(t_j)\right\|_{L^{2d/(d-1)}}\le
(c_{1}\left\| \nabla\bu(t_j)\right\|_0\left\| \bu(t_j)\right\|_2)^{1/2}\le (c_{1}M_1M_2)^{1/2}$. The proof is finished
by writing $\tilde \bv_h^j-I_h(\bu(t_j))=(\tilde \bv_h^j-\bu(t_j))
+(\bu(t_j)-I_h(\bu(t_j)))$ and applying~(\ref{cota_inter}) and~(\ref{cota_th1_h1_bis}).

For the second bound we observe that from \cite[Lemma 4.4]{heyran1} it follows $\|\bw_h\|_\infty\le C h^{-1/2}\|\nabla \bw_h\|_1
\left|\Omega\right|^{(3-d)/(2d)}$,
where $C$ depends on~$c_{\mathrm{inv}}$ and~$c_A$, and then we argue as before.
\end{proof}

In the sequel, for sequences~$(\bw_h^n)_{n=0}^N$ of $N+1$ terms
in~$V_h$ we denote
\begin{align*}
\nrm (\bw_h^n)_{n=0}^N \nrm_{\delta,\Delta t}
=&\max_{0\le n \le N}
\Bigl(
\|\bw_h^{n}\|_0^2+\Delta t\sum_{j=0}^{n-1}\bigl(\nu\|\nabla\bw_h^{j+1}\|_0^2+{\delta}\|\nabla Z_h\bw_h^{j+1}\|_0^2\bigr)
\Bigr)^{1/2},
\end{align*}
where the mapping $Z_h:V_h\rightarrow Q_h$ is defined for~every~$\bw_h\in V_h$ as the solution of
$$
\delta(\nabla Z_h\bw_h,\nabla \psi_h)=
-(\nabla \cdot  \bw_h,\psi_h),\qquad \forall \psi_h\in Q_h.
$$


The following result establishes the stability
of~discretization~(\ref{eq:ns_non_tilde}) restricted to $h$-dependent
thresholds, a concept due to L\'opez-Marcos and Sanz-Serna~\cite{LPSS2} (see also \cite{LPSS3}).

\begin{lema}\label{le:nonlin_stab}
Fix~$\Gamma_1>0$, $\rho_1>0$ and $\Lambda\ge 1$, and let $(\tilde \bv_h^n,q_h^n)$  be the solution
of (\ref{eq:eu_non_tilde})-(\ref{eq:eu_non_tilde2}) with $\bg$ defined in (\ref{eq:ns_lag}) and initial condition $(\bv_h^0,q_h^0)=({\mbf s}_h^0,z_h^0)$.
Then, under  the assumptions of Lemma~\ref{le:cotas_threshold},
there exist positive constants~$h_0$ and $S$ (the stability constant given by~(\ref{eq:constat_S}) below) such that for any~$h\le h_0$, and any two sequences
 $(\bw_{1,h}^n)_{n=1}^N$ and $(\bw_{2,h}^n)_{n=0}^N$ in~$V_h$ satisfying the threshold condition
\begin{equation}
\label{eq:threshold}
\biggl(\Delta t\sum_{j=0}^N\nu \|\nabla(\bw_{i,h}^j - \tilde \bv_h^j)\|_0^2\biggr)^{1/2}\le {\Gamma_1} h^{1/2}, \quad i=1,2,\quad\Delta t=T/N.
\end{equation}
the following bound holds
\begin{align}\label{eq:stab_ns}
  \nrm (\bw_h^n)_{n=0}^N \nrm_{\delta,\Delta t}
&\le S\Bigl(\|\bw_h^{0}\|_0^2+ \Delta t\sum_{j=1}^{N}\frac{1}{\nu}\|{\bs\tau}_h^j\|_{-1}^2\Bigr)^{1/2},
\end{align}
where, for~$n=0,1,\ldots,N$, $\bw_h^n=\bw_{h,1}^n-\bw_{h,2}^n$, and ${\bs\tau}_h^n$ is defined by
\begin{align*}
({\bs\tau}_h^n,{\bs\chi}_h)=&
\Bigl(\frac{\bw_h^{n}-\bw_h^{n-1}}{\Delta t},{\bs\chi}_h\Bigr)+\nu (\nabla \bw_h^{n},\nabla {\bs\chi}_h)+(\nabla Z_h\bw_h^{n-1},{\bs\chi}_h)\nonumber\\
&{}+\bigl(B(\bw_{1,h}^{n-1},\bw_{1,h}^{n-1})-B(\bw_{2,h}^{n-1},\bw_{2,h}^{n-1}),{\bs\chi}_h\bigr),
{}\qquad\quad{\bs\chi}_h\in V_h.\nonumber\\
\end{align*}
\end{lema}
\begin{proof} Applying~(\ref{eq:stab_evol_1}) with $d_h^n=0$ and
$$
\bb_h^n={\bs\tau}_h^{n+1} - \bigl(B(\bw_{1,h}^{n},\bw_{1,h}^{n})-B(\bw_{2,h}^{n},\bw_{2,h}^{n})\bigr),
$$
for~$n=0,1,\ldots,N,$ we have that the left hand side of~(\ref{eq:stab_ns})
can be bounded by
\begin{equation}\label{eq:sum_bb}
\|\bw_h^{0}\|_0^2+ \Delta t\sum_{j=1}^{n}\frac{1}{\nu}\|{\bs\tau}_h^j\|_{-1}^2 + \Delta t\sum_{j=0}^{n-1}\frac{1}{\nu}
\bigl\|
B(\bw_{1,h}^{j},\bw_{1,h}^{j})-B(\bw_{2,h}^{j},\bw_{2,h}^{j})\bigr\|_{-1}^2.
\end{equation}
We will now show that for some  positive $\gamma_0,\ldots,\gamma_{N-1}$ and $L>0$ satisfying
$$
\Delta t\sum_{j=0}^{n-1}\gamma_j \le L,
$$
 the last sum in~(\ref{eq:sum_bb}) can be bounded as
\begin{equation}
\label{eq:cota_con_L}
\Delta t\sum_{j=0}^{n-1}\frac{1}{\nu}
\bigl\|
B(\bw_{1,h}^{j},\bw_{1,h}^{j})-B(\bw_{2,h}^{j},\bw_{2,h}^{j})\bigr\|_{-1}^2\le \Delta t \frac{1}{\nu}\sum_{j=0}^{n-1}
\gamma_j\|
\bw_{h}^{j}\|_{0}^2,
\end{equation}
so that applying~Lemma~\ref{le:gronwall} the proof will be finished.
We do this for the more difficult case $d=3$. For ${\bs\phi}\in H^1_0(\Omega)^3$, applying~(\ref{eq:nonlin2}) we have
\begin{align*}
\bigl(
B(\bw_{1,h}^{j},\bw_{1,h}^{j})-B(\bw_{2,h}^{j},&\bw_{2,h}^{j}),{\bs\phi}\bigr)
\le
\|\bw_h^j\|_{0}\Bigl(\bigl(\|\bw_{1,h}^j\|_\infty
+\|\bw_{2,h}^j\|_\infty\bigr)\left\|\nabla{\bs\phi}\right\|_0\
\nonumber\\
&{}+\bigl(\|\nabla\bw_{1,h}^j\|_{L^{2d/(d-1)}} +
\|\nabla \bw_{2,h}^j\|_{L^{2d/(d-1)}}\bigr) \left\|{\bs\phi}\right\|_{L^{2d}}
\Bigr).
\end{align*}
Applying~Sobolev's inequality we have that~$\left\|{\bs\phi}\right\|_{L^{2d}}
\le c_{1}\left\|{\bs\phi}\right\|_1$, so that, we can take
$$
\gamma_j = 2\bigl(\|\bw_{1,h}^j\|_\infty^2
+\|\bw_{2,h}^j\|_\infty^2\bigr) + c_1^2 \bigl(\|\nabla\bw_{1,h}^j\|_{L^{2d/(d-1)}} +
\|\nabla \bw_{2,h}^j\|_{L^{2d/(d-1)}}\bigr),
$$
which, in view
of~Lemma~\ref{le:cotas_threshold} and the threshold condition~(\ref{eq:threshold}) we see that~(\ref{eq:cota_con_L})
follows with
$$
L=4\bigl(c_{\mathrm{r}}\Gamma_1^2\nu^{-1}\Lambda+ C_{\mathrm{th}}\bigr)(1+c_{1}^2).
$$
Thus, we have that the statement of the Lemma holds with
\begin{equation}
\label{eq:constat_S}
S=\exp(L/\nu).
\end{equation}
\end{proof}

To proof the convergence of the numerical approximation $(\tilde\bu_h^n)_{n=0}^N$
we will apply the following result due to Stetter \cite[Lemma~1.2.2]{Stetter}.

\begin{lema}\label{le:Stetter}
Let $(X,\left\|\cdot\right\|_X)$ and $(Y,\left\|\cdot\right\|_Y)$ be two normed linear spaces with
the same finite dimension. Let
$F:X\rightarrow Y$ be a mapping continuous in $B_X(x_u,r_1)= \{ x\in X\mid \left\|x-x_u\right\|_X<r_1\}$,
for which  there exist $S>0$ and $r_2>0$ such that
\begin{equation}
\label{eq:Stetter_st}
\left\|x_1-x_2\right\|_X\le S\left\| F(x_1)-F(x_2)\right\|_Y,
\end{equation}
for every $x_1,x_2\in B_X(x_u,r_1)$ satisfying $\left\| F(x_j)-F(x_u)\right\|_Y \le r_2$, for $j=1,2$.

Then, for $r_0=\min(r_2,r_1/S)$ the mapping $F^{-1}$ exists and is Lipschitz-continuous in $B_Y(F(x_u),r_0)$ with
Lipschitz constant equal to~$S$.
\end{lema}

Before applying Lemma~\ref{le:Stetter} we need to prove a consistency result.

\begin{lema}\label{le:semi_cons} Let $(\bu,p)$ be the solution of (\ref{NS}) and let $(\tilde \bv_h^n,q_h^n)$  be the solution
of (\ref{eq:eu_non_tilde})-(\ref{eq:eu_non_tilde2}) with $\bg$ defined in (\ref{eq:ns_lag}) and $(\tilde \bv_h^0,q_h^0)=({\mbf s}_h^0,z_h^0)$. Then, there exists a positive constant~$C_B$, depending on the Sobolev's constant~$c_{1}$,
$c_\textrm{\rm A}$ in~(\ref{eq:agmon}), $C_{th}$ in~Lemma~\ref{le:cotas_threshold}, $M_1$, $M_2$ in (\ref{eq:M1-M2})
and the ratio~$\Lambda$ in~(\ref{eq:quasi}),  such that the truncation error
\begin{eqnarray*}
{\bs\tau}_h^n=
P_{V_h}\bigl(B(\tilde\bv_{h}^{n-1},\tilde\bv_{h}^{n-1})-B(\bu(t_{n}),\bu(t_{n}))\bigr),\qquad n=0,\ldots, N-1
\end{eqnarray*}
satisfies the following bounds
\begin{align*}
\Delta t \sum_{j=1}^n\|{\bs \tau}_h^j\|_{0}^2&\le C_B^2\biggl(
\Delta t\sum_{j=0}^{n-1}\|\nabla(\tilde\bv_h^{j}-\bu(t_{j}))\|_0^2 +\nu^2 \Delta t^2K_{3,2}^2
\biggr),
\nonumber\\
\Delta t \sum_{j=1}^n\|{\bs \tau}_h^j\|_{-1}^2&\le C_B^2\left(t_n\max_{0\le j\le n-1}\|\tilde\bv_h^{j}-\bu(t_{j})\|_0^2+
\nu^2 \Delta t^2 K_{2,2}^2
\right).
\nonumber
\end{align*}

\end{lema}
\begin{proof} We concentrate on the more difficult case $d=3$.
We write ${\bs \tau}_h^n={\bs \tau}_{1,h}^n+{\bs \tau}_{2,h}^n$, were
\begin{eqnarray*}
{\bs \tau}_{1,h}^n &=& P_{V_h}\bigl(B(\tilde\bv_{h}^{n-1},\tilde\bv_{h}^{n-1})-B(\bu(t_{n-1}),\bu(t_{n-1}))\bigr),\\
{\bs \tau}_{2,h}^n &=& P_{V_h}\bigl(B(\bu(t_{n-1}),\bu(t_{n-1}))-B(\bu(t_{n}),\bu(t_{n}))\bigr).
\end{eqnarray*}
For~${\bs \tau}_{1,h}^n$, applying
Lemma~\ref{le:nonlin} and denoting $\be = \tilde \bv_h^{n-1}-\bu(t_{n-1})$, we have that it can be bounded by
\begin{align*}
\| {\bs\tau}_{1,h}\|_0\le &\left\|\be\right\|_{L^{2d}}\bigl(\|\nabla\bu(t_{n-1})\|_{L^{2d/(d-1)}}+
\|\nabla\tilde \bv_h^{n-1}\|_{L^{2d/(d-1)}}\bigr)\\
&{} +
\left\|\nabla\be\right\|_0\bigl(\left\|\bu(t_{n-1})\right\|_\infty
+\left\|\tilde \bv_h^{n-1}\right\|_\infty\bigr).
\end{align*}
Arguing similarly with~${\bs \tau}_{2,h}^n$, and denoting $\hat\be = \tilde \bu(t_{n-1})-\bu(t_{n})$, it can be bounded by
\begin{align*}
\| {\bs\tau}_{2,h}\|_0 \le &
\|\hat \be\|_{L^{2d}}\bigl(\|\nabla\bu(t_n)\|_{L^{2d/(d-1)}} +
\|\nabla\bu(t_{n-1})\|_{L^{2d/(d-1)}}\bigr)\\
&{}+
\left\|\nabla\hat\be\right\|_0\bigl(\left\|\bu(t_n)\right\|_\infty
+\left\|\bu(t_{n-1})\right\|_\infty\bigr).
\end{align*}
Applying Agmon's inequality (\ref{eq:agmon}), (\ref{eq:parti_ineq}) and~Lemma~\ref{le:cotas_threshold},
and noticing that due to H\"older's inequality we can write
$$
\left\|\nabla\hat\be\right\|_0=\biggl\|\int_{t_{n-1}}^{t_{n}} \nabla\bu_t(t)\,dt\biggr\|_0 \le \Delta t^{1/2}
\biggl(\int_{t_{n-1}}^{t_n} \|\nabla \bu_t\|_0^2\,dt\biggr)^{1/2},
$$
then it follows that
$$
\|{\bs \tau}_h^n\|_{0}\le C_B^0\left(\|\nabla(\tilde\bv_h^{n-1}-\bu(t_{n-1}))\|_0
+ \Delta t^{1/2} \biggl(\int_{t_{n-1}}^{t_n} \|\nabla \bu_t\|_0^2\,dt\biggr)^{1/2}\right),
$$
where
\begin{equation}
\label{eq:constantCB}
C_B^0 = 2(c_{\textrm{A}}+c_{1}^{3/2})(M_1M_2)^{1/2}
+C_{th}(1+c_{1}).
\end{equation}
Recalling now the definition of $K_{3,2}$ in~(\ref{eq:u-int}),
the bound for the $L^2$ norm follows with constant~$C_B=C_B^0\sqrt{2}$.

To prove the estimate in the negative norm, we recall that due to Sobolev's inequality we have that for ${\bs \phi}\in H^1(\Omega)^3$, we have that
$\left\|P_{V_h}{\bs\phi}\right\|_{L^{6}}\le c_{1}\left\|P_{V_h}{\bs\phi}\right\|_1$. Taking into account
that $\|P_{V_h}{\bs\phi}\|_1\le C\|{\bs\phi}\|_1$   from~(\ref{eq:nonlin2}) it follows that
\begin{align*}
\|{\bs \tau}_{1,h}^n\|_{-1}
\le &C\left\|\be\right\|_0\bigl(c_{1}\left\|\nabla\bu(t_{n-1})\right\|_{L^{2d/(d-1)}}
+\left\|\bu(t_{n-1})\right\|_{\infty}\bigr)\\
&{} +\left\|\be\right\|_0
\bigl(\|\tilde\bv_h^{n-1}\|_\infty+c_{1}\|\nabla\tilde\bv_h^{n-1}\|_{L^{2d/(d-1)}}
\bigr),
\end{align*}
and a similar result for~${\bs \tau}_{2,h}^n$ with~$\tilde\bv_h^{n-1}$ replaced by $\bu (t_n)$,
from where the result for the negative norm follows easily with a constant $C_B^1$ proportional  to $C_B^0$ in~(\ref{eq:constantCB}).
The proof of the lemma finishes taking $C_B=\max(C_B^0,C_B^1)$.
\end{proof}

\begin{theorem}\label{th:ns_conv} Under  the assumptions of Lemma~\ref{le:cotas_threshold}, assuming also~(\ref{eq:cota_deltat_up}),  then
 the solution $(\tilde \bu_h^n,p_h^n)$
of~(\ref{eq:ns_non_tilde}) with initial condition $(\tilde \bu_h^0,p_h^0)=({\mbf s}_h^0,z_h^0)$ satisfies the following bounds
for $n=1,\ldots,N$ and for $h$ small enough:
\begin{eqnarray}\label{eq:ns_conv0}
t_n\|\tilde \bu_h^n-\bu(t_n)\|_0^2&\le& \hat C_1 t_n\Delta t^2 + \hat C_2t_n(h^4+(\nu\delta)^2))
\end{eqnarray}
where
\begin{align}\label{laC1hat}
\hat C_1 &= \bigl( 1+(SC_B)^2T\nu^{-1}\bigr) C_1 + \nu (SC_B)^2K_{2,2}^2\\
\hat C_2 & = \bigl( 1+(SC_B)^2T\nu^{-1}\bigr) C_2,
\label{laC2hat}\
\end{align}
$C_1$ and~$C_2$ being the constants in~(\ref{laC1}) and~(\ref{laC2}).
Also, assuming for simplicity that (\ref{delta<T}) holds,
\begin{equation}\label{eq:ns_conv1}
\begin{split}
\Delta t\sum_{j=1}^{n}\bigl(\nu\|\nabla (\tilde \bu_h^j&-\bu(t_j))\|_0^2+\delta \|\nabla (q_h^j-q(t_j))\|_0^2\bigr)
\\
&\le \tilde C_1\Delta t + \tilde C_2(h^2+\nu\delta) +  \hat C_1
\Delta t^2 + \hat C_2(h^4+(\nu\delta)^2),
\end{split}
\end{equation}
where $\tilde C_1$ and $\tilde C_2$ are the constants in (\ref{laCtilde1}) and (\ref{laCtilde2}).
\end{theorem}
\begin{proof}
We apply Lemma~\ref{le:Stetter} with $X=Y=V_h^{N+1}$, $\left\| \cdot\right\|_{X}=\nrm\cdot\nrm_{\delta,\Delta t}$,
$$
\| ({\bs \tau}_h^n)_{n=0}^N\|_{Y} =
\Bigl(\|{\bs \tau}_h^{0}\|_0^2+ \Delta t\sum_{j=1}^{N}\frac{1}{\nu}\|{\bs \tau}_h^j\|_{-1}^2\Bigr)^{1/2}.
$$
and~$x_u=(\tilde\bv_h^n)_{n=0}^N$ where $(\tilde \bv_h^n,q_h^n)$  is the solution
of (\ref{eq:eu_non_tilde})-(\ref{eq:eu_non_tilde2}) with $\bg$ defined in (\ref{eq:ns_lag})
and initial condition $(\tilde \bv_h^0,q_h^0)=({\mbf s}_h^0,z_h^0)$.
We take $r_2=1$ and $r_1=\Gamma_1h^{1/2}$, and $F$ defined by
$F((\bw_h^n)_{n=0}^N)=(\bF_h^n)_{n=0}^N$ where
\begin{equation}
\label{eq:F_0}
\bF_h^0= \bw_h^0-\tilde \bv_h^0,
\end{equation}
and
for $n=1,\ldots,N$,
$\bF_h^n$ is the element in~$V_h$ satisfying
\begin{align}
\label{eq:F_n}
(\bF_h^n,{\bs\chi}_h)=&
\Bigl(\frac{\bw_h^{n}-\bw_h^{n-1}}{\Delta t},{\bs\chi}_h\Bigr)+\nu (\nabla \bw_h^{n},\nabla {\bs\chi}_h)+(\nabla Z_h\bw_h^{n-1},{\bs\chi}_h)\nonumber\\
&{}+\bigl(B(\bw_{h}^{n-1},\bw_{h}^{n-1}) - \bff(t_{n})
,{\bs\chi}_h\bigr),
{}\qquad\quad{\bs\chi}_h\in V_h.
\end{align}

We notice that the truncation error~$({\bs\tau}_h^n)_{n=0}^N=
F((\tilde\bv_h^n)_{n=0}^N)$ is
\begin{align*}
{\bs\tau}_h^0&=0,\\
{\bs\tau}_h^n&=
P_{V_h}\bigl(B(\tilde\bv_{h}^{n-1},\tilde\bv_{h}^{n-1})-B(\bu(t_{n}),\bu(t_{n}))\bigr),\qquad n=1,\ldots, N.
\end{align*}
The assumption (\ref{eq:Stetter_st}) holds due to Lemma~\ref{le:nonlin_stab}.

For the mapping~$F$ defined in~(\ref{eq:F_0})-(\ref{eq:F_n}) from~Lemma~\ref{le:semi_cons}
and~(\ref{cota_th1_l2_bis}) it follows
that
$$
\| F((\tilde \bv_h^n)_{n=0}^N)\|_Y\le C_B \Bigl((T\nu^{-1})\bigl(C_1\Delta t^2 + C_2(h^4+(\nu\delta)^2)\bigr) + \nu\Delta t^2K_{2,2}^2\Bigr)^{1/2}.
$$
Then, in view of condition~(\ref{eq:cota_deltat_up}), we have
$$
\| F((\tilde \bv_h^n)_{n=0}^N)\|_Y\le C_B \Bigl((T\nu^{-1})\bigl(C_1+ C_2\nu^2)+\nu K_{2,2}^2\Bigr)^{1/2}\hbox{\rm diam}(\Omega)c_M\nu^{-1}h
+O(h^{2}).
$$
and, thus, decays faster with $h$ than $r_1=\Gamma_1h^{1/2}$. Consequently,
for $h$ sufficiently small,
the null element in~$V_h^{N+1}$ belongs to the ball centered
at~$F((\tilde \bv_h^n)_{n=0}^N)$ and with radius $\min(r_2, r_1/S)$. Observe that the null element is
the image by~$F$ of the numerical approximation with initial condition~$(\tilde \bu_h^0,p_h^0)=({\mbf s}_h^0,z_h^0)$, that is $0=F((\tilde \bu_h^n)_{n=0}^N)$. Since, according to Lemma~\ref{le:Stetter}, the
mapping~$F$ has inverse in this ball,
then the differences $\tilde {\bs\epsilon}_h^n=\tilde \bu_h^n-\tilde \bv_h^n$, $n=0,\ldots,N$  satisfy the bound
\begin{align}\label{eq:ns-conv2}
 \nrm(\tilde{\bs\epsilon}_h^n)_{n=0}^N\nrm_{\delta, \Delta t}
&\le S\| F((\tilde \bv_h^n)_{n=0}^N)-0\|_Y\nonumber\\
&\le SC_B \Bigl((T\nu^{-1})\bigl(C_1\Delta t^2 + C_2(h^4+(\nu\delta)^2)\bigr) + \nu\Delta t^2\
K_{2,2}^2\Bigr)^{1/2}
\end{align}
Let us denote by
$\tilde {\bs \epsilon}^n=\tilde \bu_h^n - \bu(t_n)$ and $\varrho^n = p_h^n -p(t_n)$,
$n=0,\ldots,N$.
The proof is finished
by writing $\tilde {\bs\epsilon}^{n}=\tilde {\bs\epsilon}_h^{n}+(\tilde\bv_h^n-\bu(t_n))$, and~$\varrho_h^n=\varrho^n + (q_n^n-p(t_n))$, $n=0,1,\ldots,N$
and applying the bounds (\ref{cota_th1_l2_bis}) and (\ref{cota_th1_h1_bis}).
\end{proof}


\begin{remark} \label{re:full-implicit} Although we have analyzed  a semi implicit method, the analysis, with some minor changes that we now comment, applies also to the fully implicit backward Euler method. First, using Lemma~5.1 in~\cite{heyran4} instead
of Lemma~\ref{le:gronwall}, Lemma~\ref{le:nonlin_stab} can be easily extended to the fully implicit method. Also, in the case
of the fully implicit method, the truncation error~${\bs \tau}_{
h}^n$ in Lemma~\ref{le:semi_cons} would reduce
to~${\bs \tau}_{1,h}^{n+1}$, so that the constant~$C_B$ can be taken smaller. However, in the case of the fully implicit method, existence
of the numerical solution has to be proved, but this, as the arguments leading to~(\ref{eq:ns-conv2}) above show,
would be a consequence of  the null element belonging to the ball in~$V_h^{N+1}$ centered at~$F((\tilde \bv_h^n)_{n=0}^N)$ and with radius $\min(r_2, r_1/S)$, where the inverse of~$F$ exists. Taking into account these three details, the reader will find
no difficulty in extending the results of this paper to the fully implicit method.
\end{remark}

\begin{remark} Let us observe that for $(\tilde \bu_h^n,p_h^n)$ we  take as initial condition $(\tilde \bu_h^0,p_h^0)=(\tilde \bv_h^0,q_h^0)$. Although for simplicity
we have assumed in Theorem~\ref{Th3} that $\tilde \bv_h^0={\mbf s}_h^0$, $q_h^0=z_h^0$, with $({\mbf s}_h^n,q_h^n)$ the stabilized Stokes approximation defined in (\ref{eq:pro_stokes})-(\ref{eq:pro_stokes2})
with $\bg$ defined in (\ref{eq:lag}) other initial approximations can be chosen. For example, we can choose
$\tilde \bv_h^0=I_h \bu_0$,
 the interpolant of the initial velocity,
and define $q_h^0$ as the pressure obtained solving (\ref{eq:eu_non_tilde2}) for $n=-1$. With this choice the initial pressure error $r_h^0=q_h^0-z_h^0$ satisfies
$$
(\nabla \cdot (I_h \bu_0-{\mbf s}_h^0),\psi_h)=\delta(\nabla r_h^0,\nabla \psi_h),\quad \forall \psi_h\in Q_h,
$$
from which
$\|\nabla r_h^0\|_0\le \delta^{-1}\|I_h\bu_0-{\mbf s}_h^0\|_0$. Now, since both $\delta$ and $\|I_h\bu_0-{\mbf s}_h^0\|_0$ are $O(h^2)$ we get $\|\nabla r_h^0\|_0$
is bounded by a constant. In view of (\ref{cota_th1_l2})-(\ref{cota_th1_H1}) this choice keeps the optimal rate of convergence for the approximation $(\tilde \bv_h^n,q_h^n)$ since
the term ${\Delta t}^2\|\nabla r_h^0\|_0^2$ is $O({\Delta t}^2)$.
\end{remark}
To conclude we obtain a bound for the pressure. We need a previous result that we now state
\begin{lema}\label{le:semi_cons_bis} Let $(\tilde \bu_h^n,p_h^n)$  be the solution
of ~(\ref{eq:ns_non_tilde}) with $(\tilde \bu_h^0,p_h^0)=({\mbf s}_h^0,z_h^0)$ and assume $\delta$ satisfies condition (\ref{eq:cota_deltat_up}). Then,
there exists a positive constant~$C_B^*$, depending on the Sobolev's constant~$c_{1}$,
$c_\textrm{\rm A}$ in~(\ref{eq:agmon}), $C_{th}$ and $c_{\rm r}$ in~Lemma~\ref{le:cotas_threshold}, $M_1$, $M_2$ in (\ref{eq:M1-M2}), the ratio~$\Lambda$
in~(\ref{eq:quasi}) and the constants $\hat C_1^n$ and~$\hat C_2$ in (\ref{laC1hat}--\ref{laC2hat}), such that
that the error
\begin{eqnarray}\label{tau*}
({\bs\tau}^*)_h^n=
P_{V_h}\bigl(B(\tilde\bu_{h}^{n-1},\tilde\bu_{h}^{n-1})-B(\bu(t_{n}),\bu(t_{n}))\bigr),\qquad n=1,\ldots, N
\end{eqnarray}
satisfies the following bounds
\begin{align*}
\Delta t \sum_{j=1}^n\|({\bs\tau}^*)_h^j\|_{0}^2&\le (C_B^*)^2\biggl(
\Delta t\sum_{j=0}^{n-1}\|\nabla(\tilde\bu_h^{j}-\bu(t_{j}))\|_0^2 +\nu^2 \Delta t^2K_{3,2}^2
\biggr),
\nonumber\\
\Delta t \sum_{j=1}^n\|({\bs\tau}^*)_h^j\|_{-1}^2&\le (C_B^*)^2\left(t_n\max_{0\le j\le n-1}\|\tilde\bu_h^{j}-\bu(t_{j})\|_0^2+
\nu^2 \Delta t^2 K_{2,2}^2
\right).
\nonumber
\end{align*}
%
\end{lema}
\begin{proof}
We concentrate on the more difficult case~$d=3$.
Arguing exactly as in the proof of Lemma~\ref{le:semi_cons} and using (\ref{eq:ns-conv2}) 
we have that the result for the $L^2$ norm holds with the constant~$C_B^*$ replaced by
$$
 (c_{\textrm{A}}+c_{1}^{3/2})(M_1M_2)^{1/2}
+(1+c_{1})\left(C_{th}+c_{\rm r} h^{-1}\Lambda\bigl(\hat C_1 \Delta t^2 +\hat C_2(h^4+(\nu\delta)^2)\bigr)\right),
$$
which taking into account that~$\Delta t\le \delta < Ch $ can be bounded by
\begin{equation*}
\begin{split}
(C_B^{*,0})^2 = &(c_{\textrm{A}}+c_{1}^{3/2})(M_1M_2)^{1/2}\\
&+(1+c_{1})\left(C_{th}+c_{\rm r} \Lambda\bigl(\nu^{-2}\hat C_1 +\hat C_2\bigr)c_M^2\hbox{\rm diam}(\Omega)^3\right),
\end{split}
\end{equation*}
 The result for the negative norm follows also arguing as in
Lemma~\ref{le:semi_cons}, for an appropriate constant~$C_B^{*,1}$. The proof concludes taking $C_B^*=\max(C_B^{*,0},C_B^{*,1})$.
\end{proof}
\begin{theorem}\label{th:ns_conp}
Under the assumptions of Theorem~\ref{th:ns_conv} the following bound holds
\begin{eqnarray}\label{seacabo}
\Delta t\sum_{j=1}^nt_j\|p_h^j-p(t_j)\|_0^2\le \hat C_{3}{\Delta t}  +\hat C_4(h^2+\nu\delta),
\end{eqnarray}
where~$\hat C_3$ and~$\hat C_4$  are defined by
\begin{eqnarray*}
\hat C_3 &=& C_3 + CT\bigl(c_0\lambda^{-1}
(C_B^*)^2\nu^{-1}\tilde C_1 +(C_B^*)^2 \nu^2K_{2,2}^2T)\\
&&{} + C(c_0\lambda^{-1} + \nu T)(1+(C_B^*)^2\nu^{-1}T)  \hat C_1
T,\nonumber\\
\hat C_3 &=& C_4 +CT\bigl(c_0\lambda^{-1}
(C_B^*)^2\nu^{-1}\tilde C_2\\
&&{}+C(c_0\lambda^{-1} + \nu t_{n+1})(1+(C_B^*)^2\nu^{-1}t_{n+1}) \hat C_2(\hbox{\rm diam}(\Omega^2)(1+c_M),
\nonumber
\end{eqnarray*}
where $\tilde C_1$, $\tilde C_2$, $C_3$, $C_4$, $\hat C_1$ and $\hat C_2$ and $C_6$ in (\ref{laCtilde1}), (\ref{laCtilde2}), (\ref{laC3}), (\ref{laC4}), (\ref{laC1hat}) and (\ref{laC2hat}) respectively, $C_B^*$ is the constant in~Lemma~\ref{le:semi_cons_bis},
and~$c_M$ is the constant~in~(\ref{eq:cota_deltat_up})
\end{theorem}
\begin{proof}
For the proof we argue as in the proof of Theorem~\ref{Th2}.
 We first observe  that $\tilde{\bs \epsilon}_h^n=\tilde \bu_h^n - \tilde v_h^n$ and~$\varrho_h^n=p_n^n - q_h^n$ satisfy the following relations
\begin{align}
\bigl(\frac{\tilde {\bs\epsilon}_h^{n+1}-\tilde {\bs\epsilon}_h^n}{\Delta t},{\bs\chi}_h\Bigr)&+\nu(\nabla \tilde {\bs\epsilon}_h^{n+1},\nabla {\bs\chi}_h)+
(\nabla {\varrho}_h^n,{\bs \chi}_h)=(({\bs\tau}^*)_h^{n+1},{\bs\chi}_h),
\quad \forall {\bs\chi}_h\in V_h\nonumber\\
(\nabla \cdot \tilde {\bs\epsilon}_h^{n+1},\psi_h)&+\delta(\nabla {\varrho}_h^{n+1},\nabla \psi_h)=0,\qquad \forall \psi_h\in Q_h,\label{eq:errorpre}
\end{align}
where $({\bs\tau}^*)_h^{n+1}$ is defined in (\ref{tau*}).
Applying Lemma~\ref{lema_presion} and (\ref{eq:errorpre}) it is easy to obtain
\begin{align}\label{eq:conlema1_1_bis}
\Delta t \sum_{j=1}^{n}t_{j}\|\varrho_h^{j}\|_0^2\le& C\Delta t \sum_{j=1}^{n}t_{j}\nu\delta \|\nabla \varrho_h^{j}\|_0^2
+C\Delta t \sum_{j=1}^{n}t_{j}\Bigl\|\frac{\tilde {\bs\epsilon}_h^{j+1}-\tilde {\bs\epsilon}_h^j}{\Delta t}\Bigr\|_{-1}^2
\nonumber\\
&{}+C\Delta t \sum_{j=1}^{n}t_{j}\|({\bs \tau}^*)_h^j\|_{-1}^2+C\Delta t \sum_{j=1}^{n}t_{j}\nu^2\|\nabla \tilde {\bs\epsilon}_h^{j+1}\|_0^2.
\end{align}
We will bound all the terms on the right-hand side of (\ref{eq:conlema1_1_bis}). We first observe that the first and forth terms
can be bounded by
$$
C\nu t_n\Delta t \sum_{j=0}^{n}\bigl(\nu\|\nabla\tilde{\bs\epsilon}_h^{j+1}\|_0^2+{\delta}\|\nabla \varrho_h^{j}\|_0^2\bigr)
$$
and then applying~(\ref{eq:ns-conv2}) and taking into account the value of the constants $\hat C_1$ and~$\hat C_2$
in~(\ref{laC1hat}) and~(\ref{laC2hat}) we have
 \begin{equation}\label{cota_pre_l2_1_bis}
 C\Delta t \sum_{j=1}^{n}t_{j}\nu\delta \|\nabla \varrho_h^{j}\|_0^2+C\Delta t \sum_{j=1}^{n}t_{j}\nu^2\|\nabla
 \tilde {\bs\epsilon}_h^{j+1}\|_0^2\le C\nu t_n\bigl( \hat C_1 \Delta t^2 + \hat C_2(h^4+(\nu\delta)^2)\bigr).
 \end{equation}

 To bound the third term we apply Lemma~\ref{le:semi_cons_bis} and (\ref{eq:ns_conv0}) to get
 \begin{align}\label{532}
 \Delta t \sum_{j=1}^{n}t_{j}\|({\bs \tau}^*)_h^j\|_{-1}^2&\le
  (C_B^*)^2t_n\bigl(t_n\max_{0\le j\le n-1}\|\tilde\bu_h^{j}-\bu(t_{j})\|_0^2+
\nu^2 \Delta t^2 K_{2,2}^2
\bigr)
\\
&\le (C_B^*)^2t_n\bigl(\hat C_1 t_n \Delta t^2 +\hat C_2t_n (h^4+(\nu\delta)^2)+\nu^2 \Delta t^2 K_{2,2}^2 \bigr).
\nonumber
 \end{align}

To conclude we will bound the second term on the right-hand side of (\ref{eq:conlema1_1_bis}).
Since $t_j\le t_{j+1}$ and taking into account (\ref{eq:cota_menos1}) we can write
\begin{eqnarray*}
\Delta t \sum_{j=1}^{n}t_{j}\Bigl\|\frac{\tilde {\bs\epsilon}_h^{j+1}-\tilde {\bs\epsilon}_h^j}{\Delta t}\Bigr\|_{-1}^2
&\le&  \Delta t \sum_{j=1}^{n}t_{j+1}\Bigl\|\frac{\tilde {\bs\epsilon}_h^{j+1}-\tilde {\bs\epsilon}_h^j}{\Delta t}\Bigr\|_{-1}^2\nonumber\\
&\le&
C \lambda^{-1}\Delta t \sum_{j=1}^{n}t_{j+1}\Bigl\|\frac{\tilde {\bs\epsilon}_h^{j+1}-\tilde {\bs\epsilon}_h^j}{\Delta t}\Bigr\|_{0}^2.
\end{eqnarray*}
Applying (\ref{unamas}) we get
\begin{align}
\lambda^{-1}\Delta t \sum_{j=1}^{n}t_{j+1}\Bigl\|\frac{\tilde {\bs\epsilon}_h^{j+1}-\tilde {\bs\epsilon}_h^j}{\Delta t}\Bigr\|_{0}^2&\le c_0\lambda^{-1}
\Bigl( \Delta t\sum_{j=1}^nt_{j+1}\|({\bs \tau}^*)_h^j\|_0^2\nonumber\\
&\quad+\Delta t\nu\sum_{j=1}^n\|\nabla \tilde {\bs\epsilon}_h^j\|_0^2
+\Delta t \delta \sum_{j=1}^n\|\nabla \varrho_h^j\|_0^2\Bigr).\label{cota_pre_l2_5_bis}
\end{align}
To conclude we will bound the three terms on the right-hand side of (\ref{cota_pre_l2_5_bis}).
For the first one we apply Lemma~\ref{le:semi_cons_bis} and (\ref{eq:ns_conv1})  to get
\begin{align}\label{cota_pre_l2_6_bis}
\Delta t \sum_{j=1}^{n}t_{j+1}\|({\bs\tau}^*)_h^j\|_{0}^2\le &  (C_B^*)^2t_{n+1}\biggl(
\Delta t\sum_{j=0}^{n-1}\|\nabla(\tilde\bu_h^{j}-\bu(t_{j}))\|_0^2 +\nu^2 \Delta t^2K_{3,2}^2
\biggr)
\nonumber\\
{}\le& (C_B^*)^2 \frac{t_{n+1}}{\nu}\bigl( \tilde C_1\Delta t + \tilde C_2(h^2+\nu\delta) \nonumber\\
&{}+  \hat C_1
\Delta t^2 + \hat C_2(h^4+(\nu\delta)^2)\bigr),
\end{align}
where $\tilde C_1$, $\tilde C_2$, $\hat C_1$ and~$\hat C_2$  are the constants in (\ref{laCtilde1}), (\ref{laCtilde2}),
(\ref{laC1hat}) and (\ref{laC2hat}) respectively.
Finally to bound the last two terms  on the right-hand side of (\ref{cota_pre_l2_5_bis}) we apply (\ref{eq:ns-conv2}) to obtain
\begin{eqnarray}\label{cota_pre_l2_7_bis}
\Delta t\nu\sum_{j=1}^n\|\nabla \tilde {\bs\epsilon}_h^j\|_0^2
+\Delta t \delta \sum_{j=1}^n\|\nabla \varrho_h^j\|_0^2\le
\hat C_1\Delta t^2 + \hat C_2 (h^4+(\nu\delta)^2).
\end{eqnarray}
Inserting (\ref{cota_pre_l2_6_bis}) and (\ref{cota_pre_l2_7_bis}) in (\ref{cota_pre_l2_5_bis}) we get
\begin{align}\label{536}
\Delta t \sum_{j=1}^{n}t_{j}\Bigl\|\frac{\tilde {\bs\epsilon}_h^{j+1}-\tilde {\bs\epsilon}_h^j}{\Delta t}\Bigr\|_{-1}^2&
\le c_0\lambda^{-1}
\Bigl((C_B^*)^2\nu^{-1}t_{n+1}\bigl( \tilde C_1\Delta t + \tilde C_2(h^2+\nu\delta)\bigr)
\\
&{}+\bigl(1+(C_B^*)^2\nu^{-1}t_{n+1}\bigr)\bigl(\hat C_1
\Delta t^2 + \hat C_2(h^4+(\nu\delta)^2)\bigr)\Bigr).
\nonumber
\end{align}
Inserting (\ref{cota_pre_l2_1_bis}), (\ref{532}) and (\ref{536}) into (\ref{eq:conlema1_1_bis}) we get
\begin{eqnarray*}
&&\Delta t \sum_{j=1}^{n}t_{j}\|\varrho_h^{j}\|_0^2 \le
Cc_0\lambda^{-1}
(C_B^*)^2\nu^{-1}t_{n+1}\bigl( \tilde C_1\Delta t + \tilde C_2(h^2+\nu\delta)\bigr)
\nonumber\\
&&\qquad\quad{}+C(c_0\lambda^{-1} + \nu t_{n+1})(1+(C_B^*)^2\nu^{-1}t_{n+1}) \bigl( \hat C_1
\Delta t^2 + \hat C_2(h^4+(\nu\delta)^2)\bigr)
\nonumber\\
&&\qquad\quad {}+C(C_B^*)^2 t_n\nu^2K_{2,2}^2 \Delta t^2.
\end{eqnarray*}
Applying triangle inequality together with (\ref{cota_prefinal_mod}) we finally reach (\ref{seacabo}).
\end{proof}

\end{document}